\documentclass[12pt,a4paper]{amsproc}
\usepackage[english]{babel}
\usepackage{amssymb}
\usepackage{amscd}
\usepackage[mathscr]{eucal}
\usepackage{amssymb,amsmath,amsthm, accents}
\usepackage{tikz-cd}
\usepackage{pgfplots,tikz}
\usetikzlibrary{decorations.pathmorphing,patterns}
\usepackage{afterpage}
\usepackage{extarrows}

\usepackage{geometry}
\newtheorem{thm}{Theorem}[section]

\newtheorem{prop}[thm]{Proposition}

\theoremstyle{remark}
\newtheorem{rem}[thm]{Remark}

\theoremstyle{definition}
\newtheorem{defi}[thm]{Definition}

\newcommand{\cL}{\mathcal{L}}
\newcommand{\cN}{\mathcal{N}}

\renewcommand\Im{\operatorname{Im}}

\newcommand{\Ext}{\operatorname{Ext}}

\newcommand{\QBan}{\operatorname{\textbf{QB}}}
\newcommand{\pBan}{\operatorname{\textbf{pB}}}
\newcommand{\Ban}{\operatorname{\textbf{B}}}

\newcommand{\Id}{\operatorname{Id}}

\newcommand\restr[2]{{
		\left.\kern-\nulldelimiterspace 
		#1 
		\right|_{#2} 
}}
\newcommand{\norm}[1]{{\left\vert\kern-0.25ex\left\vert #1 
		\right\vert\kern-0.25ex\right\vert}}
\newcommand{\nnorm}[1]{{\left\vert\kern-0.25ex\left\vert\kern-0.25ex\left\vert #1 
		\right\vert\kern-0.25ex\right\vert\kern-0.25ex\right\vert}}

\numberwithin{equation}{section}

\usepackage[colorlinks=true,
            linkcolor=blue,
            citecolor=blue,
            urlcolor=blue]{hyperref}
            
\begin{document}
	\title[Universal $\delta$-functors and relative derivation in $\pBan$]{Universal $\delta$-functors and relative injective objects \\ in the category of $p$-Banach spaces}
	
	\author[N. Trejo-Arroyo]{Nazaret Trejo-Arroyo}
	\address{Universidad Complutense de Madrid\\ Plaza de las Ciencias, s/n \\ 28040-Madrid\\ Spain} \email{ntrejo@ucm.es}

    \thanks{This research has been supported by a predoctoral grant associated to ``Programa de financiación de UCM-Banco Santander (CT24/25)''}
	\maketitle
\begin{abstract}
Inspired by Grothendieck's foundational work \cite{tohoku}, we establish that the sequence of functors $(\Ext^n_{\pBan}(E,-))_{n \ge 0}$ on the category of $p$-Banach spaces ($0 < p < 1$) forms a universal $\delta$-functor. To this end, for any pair of $p$-Banach spaces $E$ and $X$, we construct a  $p$-Banach space $\mathcal N_E^X$ that renders each functor $\Ext^n_{\pBan}(E,-)$ \emph{effaceable}. Furthermore, we show that this space serves as a relative injective object for $E$ and $X$, enabling the construction of relative injective presentations and proving that the functors $\Ext^n_{\pBan}(E,-)$ are the right derived functors of $\mathcal{L}(E,-)$. We will also examine the situation in the category of quasi-Banach spaces.
\end{abstract}  
	\section{Introduction}
	The present paper deals with the derived functors of the covariant functors $\mathcal L_{\pBan}(E,-)$ (of linear and continuous operators) defined on the category of $p$-Banach spaces ($0<p<1$). It is an established fact that the are no non-zero injective objects in the category of $p$-Banach spaces and operators (see \cite[2.9.1]{hmbst}). Under such conditions we cannot carry out the classical derivation process of a covariant, left-exact functor defined in $\pBan$. However, for any $n\in \mathbb N$, we can always consider the Yoneda (vector) spaces $\Ext^n_{\pBan}(X,Y)$ for any given $p$-Banach spaces $X$ and $Y$, and construct the corresponding ``long homology sequence'' associated to the functor $\mathcal{L}_{\pBan}(E,-)$, for $E$ a fixed $p$-Banach space. That is, given a short exact sequence of $p$-Banach spaces $\begin{tikzcd}[column sep=1.5em] 0 \arrow[r] & Y \arrow[r] & X \arrow[r] & Z \arrow[r] & 0\end{tikzcd}$, the following sequence
	\begin{equation}
		\label{eq:longhomology}
		\begin{tikzcd}
			& \cdots \arrow[r] 			\arrow[dl, phantom, ""{coordinate, name=Z}]                    & {\Ext_{\pBan}^n(E,Z)} \arrow[dll,
			"\omega^n" description, rounded corners=6pt,
			to path={ -- ([xshift=6.5ex]\tikztostart.east)
				|- (Z) [very near end] \tikztonodes
				-| ([xshift=-3ex]\tikztotarget.west)
				-- (\tikztotarget)}]        \\
			{\Ext_{\pBan}^{n+1}(E,Y)} \arrow[r] & {\Ext_{\pBan}^{n+1}(E,X)} \arrow[r] \arrow[dl, phantom, ""{coordinate, name=ZZ}] & {\Ext_{\pBan}^{n+1}(E,Z)} \arrow[dll,
			"\omega^{n+1}" description, rounded corners=6pt,
			to path={ -- ([xshift=6.5ex]\tikztostart.east)
				|- (ZZ) [very near end] \tikztonodes
				-| ([xshift=-3ex]\tikztotarget.west)
				-- (\tikztotarget)}]  \\
			{\Ext_{\pBan}^{n+2}(E,Y)} \arrow[r] & \cdots                              &
		\end{tikzcd}
	\end{equation}
	is exact. This fact naturally leads one to ask in which sense the functors $\Ext^n_{\pBan}(E,-)$ could be considered the right derived functors of $\mathcal{L}_{\pBan}(E,-)$. 
	\par Our main objective is to construct, for $E$ and $X$ two fixed $p$-Banach spaces, another $p$-Banach space $\mathcal N_E^X$ which behaves like an injective object for these two particular spaces (see Section \ref{main} below). By means of such objects, we then adapt the classical derivation process of the functor $\mathcal{L}_{\pBan}(E,-)$ and recover the functors $\Ext^n_{\pBan}(E,-)$ as derived functors in Section \ref{derived}. The construction of the space $\mathcal N_E^X$ follows the spirit of Grothendieck's paper \cite{tohoku} and the notion of \emph{universal $\delta$-functor}, as we make explicit in Section \ref{grothendieck}.

    \par On the other hand, we can also constrct a sequence of the form
    \begin{equation}
		\begin{tikzcd}
			& \cdots \arrow[r] 			\arrow[dl, phantom, ""{coordinate, name=Z}]                    & {\Ext_{\QBan}^n(E,Z)} \arrow[dll,
			"\omega^n" description, rounded corners=6pt,
			to path={ -- ([xshift=6.5ex]\tikztostart.east)
				|- (Z) [very near end] \tikztonodes
				-| ([xshift=-3ex]\tikztotarget.west)
				-- (\tikztotarget)}]        \\
			{\Ext_{\QBan}^{n+1}(E,Y)} \arrow[r] & {\Ext_{\QBan}^{n+1}(E,X)} \arrow[r] \arrow[dl, phantom, ""{coordinate, name=ZZ}] & {\Ext_{\QBan}^{n+1}(E,Z)} \arrow[dll,
			"\omega^{n+1}" description, rounded corners=6pt,
			to path={ -- ([xshift=6.5ex]\tikztostart.east)
				|- (ZZ) [very near end] \tikztonodes
				-| ([xshift=-3ex]\tikztotarget.west)
				-- (\tikztotarget)}]  \\
			{\Ext_{\QBan}^{n+2}(E,Y)} \arrow[r] & \cdots                              &
		\end{tikzcd}
	\end{equation}
in the category of quasi-Banach spaces. Therfore, the natural question is now the same: in which sense the functors $\Ext^n_{\QBan}(E,-)$ are the derived functors of $\mathcal{L}_{\QBan}(E,-)$? However, in $\QBan$ the scene is even worse, since there are neither projective nor injective objects. We analyze the homological behaviour of the category $\QBan$ in Section \ref{sec:QBan}. 

	\section{Preliminaries}
    \subsection{Our categories} We will mainly deal with the categories $\boldsymbol{\pBan}$, for $0 <p \leq 1$, whose objects are $p$-Banach spaces and whose morphisms are (linear, continuous) operators, and $\boldsymbol{\QBan}$, whose objects are quasi-Banach spaces and whose morphisms are again operators. Given $X$ and $Y$ $p$-Banach spaces (respectively, quasi-Banach spaces), we will denote $\mathcal{L}_{\boldsymbol{\pBan}}(X,Y)$ (respectively,  $\mathcal{L}_{\boldsymbol{\QBan}}(X,Y)$) the collection of operators from $X$ to $Y$. For $p=1$, $\boldsymbol{\pBan}$ is just $\boldsymbol{\Ban}$, the category of Banach spaces and operators. It is an important and useful remark that every $p$-Banach space is also a $r$-Banach spaces, for $0<r<p \leq 1$. Moreover, it holds that every $p$-Banach space is a quasi-Banach space and, given a quasi-Banach space $X$, the classical result of Aoki and Rolewicz \cite[Theorem 1.2]{fspace} guarantees that there always exists $p\in (0,1]$ such that the quasi-norm defined on $X$ is equivalent to a $p$-norm. For this reason, we will focus primarily on the category $\boldsymbol{\pBan}$ of $p$-Banach spaces until Section~\ref{sec:QBan}, which is dedicated to quasi-Banach spaces. Nevertheless, all the constructions introduced in Sections~\ref{sec:pre_2.2}--\ref{sec:pre_2.5} can also be carried out in $\boldsymbol{\QBan}$.

\subsection{The quasi-abelian structure of \texorpdfstring{$\boldsymbol{\pBan}$}{pBan}} \label{subsec:quasi-abelian} \label{sec:pre_2.2}

We recall that an additive category with finite kernels and cokernels is \emph{quasi-abelian} if strict monomorphisms are stable under pushouts and strict epimorphisms are stable under pullbacks (see \cite{buehler2010, schneiders1999} for details). Recall also that, in a quasi-abelian category, a morphism $T: X \to Y$ is called a \emph{strict morphism} if, and only if, the canonical map $\overline{T}: \mathrm{Coim}(T) \to \mathrm{Im}(T)$ is an isomorphism, where $\mathrm{Coim}(T) = \mathrm{coker}(\ker(T))$ and $\mathrm{Im}(T) = \ker(\mathrm{coker}(T))$. 

Furthermore, $T$ is a \emph{strict monomorphism} if, and only if, $T$ is a strict morphism and $\ker(T) = 0$, so $T \cong \mathrm{Im}(T)$. On the other hand, $T$ is a \emph{strict epimorphism} if, and only if, $T$ is a strict morphism and $\mathrm{coker}(T) = 0$, so $\mathrm{Coim}(T) \cong Y$.
\par Moreover, a sequence
\begin{equation}\notag
    \begin{tikzcd}
		0 \arrow[r] & Y \arrow[r, "i"] & X \arrow[r, "q"] & Z \arrow[r] & 0
	\end{tikzcd} 
    \end{equation}
    in a quasi-abelian category $\mathcal{C}$ is called a \emph{strict short exact sequence} if, and only if, the kernel of each morphism coincides with the range of the previous one, $i$ is a strict monomorphism and $q$ is a strict epimorphism. In the category $\boldsymbol{\pBan}$, for $0 < p \le 1$, the quasi-abelian structure manifests as follows:

\begin{prop}\label{prop:pban_quasi_abelian}
Let $T: X \to Y$ be a bounded linear operator in $\boldsymbol{\pBan}$.
\begin{itemize}
    \item[\emph{i)}] The kernel of $T$ is the inclusion of the closed subspace $\ker(T) = T^{-1}(0) \hookrightarrow X$, while the cokernel of $T$ is the canonical projection $Y \to Y / \overline{T(X)}$.
    \item[\emph{ii)}] The map $T$ is a \emph{strict monomorphism} if, and only if, $T$ is an embedding with closed image.
    \item[\emph{iii)}] The map $T$ is a \emph{strict epimorphism} if, and only if, $T$ is surjective. Recall that, due to the Open Mapping Theorem, surjectivity implies that $T$ is an open map.
\end{itemize}
\end{prop}
\begin{proof} Let us check \emph{i)}. The space $\ker(T) = T^{-1}(0)$ is a closed subspace of $X$ because $T$ is continuous and $\{0\}$ is closed. Thus, $T^{-1}(0)$ is itself a $p$-Banach space. Let $i: T^{-1}(0) \hookrightarrow X$ denote the canonical inclusion. By construction, $T \circ i = 0$. If $S: Z \to X$ is any bounded linear operator satisfying $T \circ S = 0$, then $S(Z) \subseteq T^{-1}(0)$, which implies that $S$ factors uniquely through $i$ via a bounded operator $\widetilde{S}: Z \to T^{-1}(0)$ such that $S = i \circ \widetilde{S}$. Hence, $i = \ker(T)$. On the other hand, for the cokernel of $T$, let $N = \overline{T(X)}$. Since $N$ is a closed subspace, the quotient space $Y/N$ (equipped with the quotient $p$-norm) is complete and the canonical projection $\pi: Y \to Y/N$ is a bounded operator with $\pi \circ T = 0$. Suppose $R: Y \to Z$ is a bounded linear operator such that $R \circ T = 0$. Then $T(X) \subseteq \ker(R)$. Because $\ker(R)$ is closed in $Y$, it follows that $N = \overline{T(X)} \subseteq \ker(R)$. By the universal property of quotients of topological vector spaces, $R$ factors uniquely through $\pi$ via a bounded linear operator $\overline{R}: Y/N \to Z$ such that $R = \overline{R} \circ \pi$. Consequently, $\pi = \mathrm{coker}(T)$.
    \par For \emph{ii)}, suppose $T$ is a strict monomorphism. Then $\overline{T}: X / \ker(T) \to \overline{T(X)}$ is an isomorphism in $\boldsymbol{\pBan}$, so that, $\ker(T) = \{0\}$. Therefore, $T$ is injective, and $X / \ker(T) \cong X$. Furthermore, since $\overline{T}$ is surjective onto $\overline{T(X)}$, we have $T(X) = \overline{T(X)}$, which proves that $T(X)$ is closed in $Y$. Because $\overline{T}: X \to T(X)$ is a topological isomorphism, $T$ is a topological embedding onto its closed image. Conversely, assume $T$ is an embedding with closed range $T(X) = \overline{T(X)}$. Therefore, $\ker(T) = \{0\}$, so $\mathrm{Coim}(T) = X$. The fact that $T(X)$ was closed implies that $\mathrm{Im}(T) = T(X)$. The canonical map $\overline{T}: X \to T(X)$ is a continuous linear bijection between complete $p$-Banach or quasi-Banach spaces. Hence, $\overline{T}^{-1}: T(X) \to X$ is bounded. Thus $\overline{T}$ is an isomorphism in $\boldsymbol{\pBan}$, proving that $T$ is a strict monomorphism.

\par Finally, we turn to \emph{iii)}. If $T$ is a strict epimorphism, the isomorphism $\overline{T}: X/\ker(T) \to Y$ is surjective, which implies that $T = \overline{T} \pi$ is surjective. Conversely, suppose $T: X \to Y$ is surjective. Then $T(X) = Y$, so $\mathrm{Im}(T) = \overline{T(X)} = Y$. The induced operator $\overline{T}: X/\ker(T) \to Y$ is a continuous linear bijection. By the Open Mapping Theorem, $\overline{T}$ is an open map, and its inverse $\overline{T}^{-1}: Y \to X/\ker(T)$ is bounded. Hence, $\overline{T}$ is an isomorphism in $\boldsymbol{\pBan}$, establishing that $T$ is a strict epimorphism.
\end{proof}


In view of Proposition \ref{prop:pban_quasi_abelian}, strict short exact sequences admit a concrete functional-analytic characterization:

\begin{defi}
	A sequence
	\[
	\begin{tikzcd}
		0 \arrow[r] & Y \arrow[r, "i"] & X \arrow[r, "q"] & Z \arrow[r] & 0
	\end{tikzcd}
	\]
	in $\boldsymbol{\pBan}$ is a \emph{strict short exact sequence} in the quasi-abelian sense --or an exact sequence in the sense of Quillen \cite{quillen1973}-- provided that 
    \begin{enumerate}
        \item the kernel of each morphism coincides with the range of the previous one, 
        \item $i$ is a topological embedding with closed image,
        \item $q$ is surjective and induces an isomorphism $X / i(Y) \cong Z$.
    \end{enumerate}  
\end{defi}
Observe, nonetheless, that Proposition \ref{prop:pban_quasi_abelian} asserts that items (2) and (3) in the previous definition are a consequence of (1). In any case, these sequences endow $\boldsymbol{\pBan}$ with their canonical exact category structures.

\subsection{Yoneda extensions}\label{sec:pre_2.3}

For $n \ge 1$, a \emph{Yoneda $n$-extension} (or a \emph{strict Yoneda exact sequence}) of length $n$ between $p$-Banach spaces is a sequence $\mathfrak{X}$ in $\boldsymbol{\pBan}$ of the form
\begin{equation} \label{next}
\mathfrak{X}: \quad
\begin{tikzcd}
	0 \arrow[r] & Y \arrow[r] & X_n \arrow[r] & X_{n-1} \arrow[r] & \cdots \arrow[r] & X_1 \arrow[r] & Z \arrow[r] & 0
\end{tikzcd}
\end{equation}
that splices into $n$ strict short exact sequences. Concretely, this means that for each $k = 1, \dots, n-1$, the intermediate image space $K_k = \ker(X_k \to X_{k-1})$ is a closed subspace (hence, a $p$-Banach space), and each induced sequence 
\[
0 \longrightarrow K_k \longrightarrow X_k \longrightarrow K_{k-1} \longrightarrow 0
\]
is a strict short exact sequence in $\boldsymbol{\pBan}$, with $K_0 = Z$ and $K_n = Y$. For $n=0$, a \emph{Yoneda $0$-extension} is defined simply as a bounded linear operator $T: Z \to Y$.
 
\par Let $n \ge 0$. A \emph{morphism} $\varphi: \mathfrak{X} \to \mathfrak{X}'$ from a Yoneda $n$-extension $\mathfrak{X}$ (from $Z$ to $Y$) to a Yoneda $n$-extension $\mathfrak{X}'$ (from $Z'$ to $Y'$) is defined as follows:
\begin{itemize}
    \item For $n \ge 1$, if $\mathfrak{X}$ and $\mathfrak{X}'$ are given by
    \begin{equation}\notag
\mathfrak{X}: \quad
\begin{tikzcd}
	0 \arrow[r] & Y \arrow[r] & X_n \arrow[r] & X_{n-1} \arrow[r] & \cdots \arrow[r] & X_1 \arrow[r] & Z \arrow[r] & 0
\end{tikzcd}
\end{equation}
\begin{equation}\notag
\mathfrak{X'}: \quad
\begin{tikzcd}
    0 \arrow[r] & Y' \arrow[r] & X'_n \arrow[r] & X'_{n-1} \arrow[r] & \cdots \arrow[r] & X'_1 \arrow[r] & Z' \arrow[r] & 0
\end{tikzcd}
\end{equation}
    a morphism $\varphi: \mathfrak{X} \to \mathfrak{X}'$ is a tuple of operators $(\varphi_{-}, \varphi_n, \dots, \varphi_1, \varphi_+)$ such that $\varphi_-: Y \to Y'$, $\varphi_k: X_k \to X'_k$ for $1 \le k \leq n$, and $\varphi_+: Z \to Z'$, making the following diagram commutative:
    \begin{equation}
		\notag
		\begin{tikzcd}
			0 \arrow[r] & Y \arrow[r] \arrow[d, "\varphi_-"] & X_n \arrow[r] \arrow[d, "\varphi_n"] & X_{n-1} \arrow[r] \arrow[d, "\varphi_{n-1}"] & \cdots \arrow[r] \arrow[d] & X_1 \arrow[r] \arrow[d, "\varphi_1"] & Z \arrow[r] \arrow[d, "\varphi_+"] & 0 \\
			0 \arrow[r] & Y' \arrow[r]                       & X'_n \arrow[r]                        & X'_{n-1} \arrow[r]                            & \cdots \arrow[r]           & X'_1 \arrow[r]                        & Z' \arrow[r]                       & 0
		\end{tikzcd}
	\end{equation}
    \item For $n = 0$, if $\mathfrak{X}$ is $T: Z \to Y$ and $\mathfrak{X}'$ is $T': Z' \to Y'$, a morphism $\varphi: \mathfrak{X} \to \mathfrak{X}'$ is a pair of continuous linear operators $(\varphi_-, \varphi_+)$ with $\varphi_-: Y \to Y'$ and $\varphi_+: Z \to Z'$ such that $T' \varphi_+=\varphi_-T$, i.e., the square
    \begin{equation} \notag
    \begin{tikzcd}
        Z \arrow[r, "T"] \arrow[d, "\varphi_+"'] & Y \arrow[d, "\varphi_-"] \\
        Z' \arrow[r, "T'"'] & Y'
    \end{tikzcd}
    \end{equation}
    commutes.
\end{itemize}
\par Moreover, for $n \geq 0$, we will say that the morphism $\varphi: \mathfrak{X} \to \mathfrak{X}'$ between Yoneda $n$-extensions has \emph{fixed ends} if:
	\begin{itemize}
		\item[(i)] $Y'=Y$ and $Z'=Z$.
		\item[(ii)] $\varphi_-=\Id_Y$ and $\varphi_+=\Id_Z$.
	\end{itemize} 
	Consider the following equivalence relation in the class of Yoneda $n$-extensions with fixed ends $Y$ and $Z$:
	\begin{equation}
		\notag
		\begin{split}
			\mathfrak{X} \sim \mathfrak{X}' \iff & \text{ there exists a sequence of } n\text{-extensions } \mathfrak{X}=\mathfrak{X}_0, \mathfrak{X}_1, \ldots, \mathfrak{X}_k=\mathfrak{X}'\\
			& \text{so that for any }  0 \leq i \leq k-1 \text{ there is either }\\
			& \text{a morphism } \mathfrak{X}_i \to \mathfrak{X}_{i+1} \text{ with } \text{fixed ends, or }\\
			& \text{a morphism } \mathfrak{X}_{i+1} \to \mathfrak{X}_{i} \text{ with fixed ends.}
		\end{split}
	\end{equation}
	It can be shown (see \cite[Corollary 6.4]{exact_categ}) that we only need two Yoneda $n$-extensions and three morphisms to establish that $\mathfrak{X} \sim \mathfrak{X}'$. We will denote $\Ext^n_{\boldsymbol{\pBan}}(Z,Y)$ to the set of equivalence classes of Yoneda $n$-extensions between $Y$ and $Z$ in $\boldsymbol{\pBan}$. We will write $\left[ \mathfrak{X} \right]$ to indicate the equivalence class of the Yoneda $n$-extension $\mathfrak{X}$ --see Equation \ref{next}-- in $\Ext^n_{\boldsymbol{\pBan}}(Z,Y)$.
	Let us observe that, for $n=0$, the general equivalence relation $\sim$ reduces to equality between operators. Indeed, if $\mathfrak{X}$ and $\mathfrak{X}'$ are two $0$-extensions, then they are given by two continuous linear operators $T:Z \to Y$ and $T':Z' \to Y'$, a morphism with fixed ends $\varphi: \mathfrak{X} \to \mathfrak{X}'$ exists if, and only if, $Y=Y'$ and $Z=Z'$, and
\[
T' \Id_Z = \Id_Y T \quad \iff \quad T' = T.
\]
Consequently, $\mathfrak{X} \sim \mathfrak{X}$ if, and only if, $ T = T'$, which establishes a canonical identification:
\[
\Ext^0_{\boldsymbol{\pBan}}(Z,Y) = \mathcal{L}_{\pBan}(Z,Y).
\]
Under this equalities, the equivalence class $[T]$ of a $0$-extension is simply the operator $T$ itself. On the other hand, for $n=1$, we obtain Yoneda $1$-extensions (more simply, \emph{Yoneda extensions}), which are precisely \emph{strict short exact sequences}
\[
	\begin{tikzcd}
		\mathfrak{X}: & 0 \arrow[r] & Y \arrow[r] & X \arrow[r] & Z \arrow[r] & 0,
	\end{tikzcd}
\]
and the equivalence relation becomes considerably simpler: two strict short exact sequences
\begin{tikzcd}[cramped, column sep=1.5em]
	\mathfrak{X}_i: 0 \arrow[r] & Y \arrow[r] & X_i \arrow[r] & Z \arrow[r] & 0
\end{tikzcd}
(for $i=1,2$) are said to be equivalent if there exists a bounded linear operator $T: X_1 \to X_2$ making the following diagram commutative:
\begin{equation}\notag
	\begin{tikzcd}
		\mathfrak{X}_1: & 0 \arrow[r] & Y \arrow[r] \arrow[d, equal] & X_1 \arrow[r] \arrow[d, "T"] & Z \arrow[r] \arrow[d, equal] & 0 \\
		\mathfrak{X}_2: & 0 \arrow[r] & Y \arrow[r] & X_2 \arrow[r] & Z \arrow[r] & 0
	\end{tikzcd}
\end{equation}
By the Short Five Lemma in quasi-abelian categories, any such operator $T$ is automatically an isomorphism in $\boldsymbol{\pBan}$.

	\subsection{Pull-back and push-out} \label{sec:pre_2.4}
    For $n \geq 1$, given a Yoneda $n$-extension in $\boldsymbol{\pBan}$
	\begin{equation}
		\notag
		\begin{tikzcd}
			0 \arrow[r] & Y \arrow[r, "j"] & X_n \arrow[r] & X_{n-1} \arrow[r] & \cdots \arrow[r] & X_1 \arrow[r, "q"] & Z \arrow[r] & 0
		\end{tikzcd}
	\end{equation}
	and an operator $T: Z' \to Z$, the \emph{pull-back Yoneda $n$-extension} induced by the couple $\lbrace q, T \rbrace$ is the lower row of the diagram 
    \begin{equation}
	\notag 
	\begin{tikzcd}
		\mathfrak{X}: & 0 \arrow[r] & Y \arrow[r, "j"] \arrow[d, equal] & X_n \arrow[r] \arrow[d, equal] & X_{n-1} \arrow[r] \arrow[d, equal] & \cdots \arrow[r] & X_1 \arrow[r, "q"] & Z \arrow[r] & 0 \\
		\mathfrak{X}T: & 0 \arrow[r] & Y \arrow[r, "j"] & X_n \arrow[r] & X_{n-1} \arrow[r] & \cdots \arrow[r] & PB \arrow[r, "\pi_2"] \arrow[u, "\pi_1"] & Z' \arrow[r] \arrow[u, "T"] & 0
	\end{tikzcd}
\end{equation}
	where the \emph{pull-back space} is defined as the space 
	$$PB=\lbrace (x,z') \in X_1 \times Z':q(x)=T(z') \rbrace$$
	The arrows $\pi_1:PB \to X_1$ and $\pi_2:PB \to Z'$ are the restrictions to $PB$ of the canonical projections. Due to the splicing process for Yoneda extensions (see next section), it can be proved that if the operator $T:Z' \to Z$ can be lifted to $X_1$, that is to say, there exists  $\hat{T}:Z' \to X_1$ such that $q \hat{T}=T$, then the pull-back Yoneda $n$-extension is trivial. The converse only holds for the case $n=1$. 
    \par For $n = 0$, given a Yoneda $0$-extension $\mathfrak{X}$ from $Z$ to $Y$ represented by an operator $S: Z \to Y$ and an operator $T: Z' \to Z$, the \emph{pull-back Yoneda $0$-extension} $\mathfrak{X}T$ is defined as the pre-composition operator $\mathfrak{X}T := ST: Z' \to Y$,
which is a Yoneda $0$-extension from $Z'$ to $Y$. 
	\par On the other hand, given a Yoneda $n$-extension
	\begin{equation}
		\notag
		\begin{tikzcd}
			0 \arrow[r] & Y \arrow[r, "j"] & X_n \arrow[r] & X_{n-1} \arrow[r] & \cdots \arrow[r] & X_1 \arrow[r, "q"] & Z \arrow[r] & 0
		\end{tikzcd}
	\end{equation} 
	in  $\boldsymbol{\pBan}$ and an operator $T: Y \to Y'$ the \emph{push-out Yoneda $n$-extension} induced by the couple $\lbrace j, T \rbrace$ is the lower row of the diagram 
	\begin{equation}
	\notag
	\begin{tikzcd}
		\mathfrak{X}: & 0 \arrow[r] & Y \arrow[r, "j"] \arrow[d, "T"] & X_n \arrow[r] \arrow[d, "i_2"] & X_{n-1} \arrow[r] \arrow[d, equal] & \cdots \arrow[r] & X_1 \arrow[r, "q"] & Z \arrow[r] & 0 \\
		T\mathfrak{X}: & 0 \arrow[r] & Y' \arrow[r, "i_1"] & PO \arrow[r] & X_{n-1} \arrow[r] & \cdots \arrow[r] & X_1 \arrow[r, "q"] \arrow[u, equal] & Z \arrow[r] \arrow[u, equal] & 0
	\end{tikzcd}
\end{equation}
	where the push-out space is defined as the space
	$$PO=\frac{Y' \oplus_p X_n}{\Delta} \quad, \quad \Delta=\lbrace (T(y), -j(y)) \in Y' \oplus_p X_n: y \in Y \rbrace$$
	The arrows $i_1: Y' \to PO$ and $i_2:X_n \to PO$ are the compositions of the natural injections into $Y' \oplus_p X_n$ with the quotient map. 
	Again, it can be proved that if the operator $T:Y \to Y'$ can be extended to $X_n$, that is to say, there exists  $\hat{T}:X_n \to Y$ such that $\hat{T}j=T$, then the push-out Yoneda $n$-extension is trivial. The converse only holds for the case $n=1$. 
    \par For $n = 0$, given a Yoneda $0$-extension $\mathfrak{X}$ from $Z$ to $Y$ represented by an operator $S: Z \to Y$ and an operator $T: Y \to Y'$, the \emph{push-out Yoneda $0$-extension} $T\mathfrak{X}$ is defined as the post-composition operator $T\mathfrak{X} := TS: Z \to Y'$,
which is a Yoneda $0$-extension from $Z$ to $Y'$.
	\par Moreover, it can be easily proved that if $\mathfrak{X}$ and $\mathfrak{X}'$ are two equivalent Yoneda $n$-extensions, then for any operator $T:E \to Z$ the \textsl{pull-back} sequences $\mathfrak{X}T$ and $\mathfrak{X}'T$ are equivalent, and for any operator $G:Y \to E$ the  \textsl{push-out} sequences $G\mathfrak{X}$ and $G\mathfrak{X}'$ are equivalent. 
	
	\subsection*{Linear structure in the class of all Yoneda $n$-extensions}
    For any positive integer $n$, the sets $\Ext^n_{\boldsymbol{\pBan}}(Z,Y)$ admit a natural linear structure whose operations are defined using pull-backs and push-outs. Indeed:
    \begin{itemize}
        \item[i)] \textbf{Sum of $n$-extensions.} Given $[\mathfrak{X}], \mathfrak{[X']} \in \Ext^n_{\pBan}(Z,Y)$ the \emph{Baer sum} $[\mathfrak{X}+\mathfrak{X'}] \in \Ext^n_{\pBan}(Z,Y)$ is constructed taking the $n$-extension $[\mathfrak{X}] \oplus_p [\mathfrak{X}']$
        \begin{equation}
            \notag
            \begin{tikzcd}
0 \arrow[r] & Y \oplus_pY \arrow[r] & X_1 \oplus_p X'_1 \arrow[r] & \cdots \arrow[r] & X_n \oplus_p X'_n \arrow[r] & Z\oplus_pZ \arrow[r] & 0
\end{tikzcd}
        \end{equation}
        and making pull-back with the \emph{diagonal} operator $\Delta: Z \to Z \oplus_p Z$, $\Delta(z) = (z,z)$ and push-out with the \emph{sum} operator $S: Y \oplus_p Y \to Y$, $S(y,y') = y+y'$:
        \begin{equation}
            \label{Baersum}
            \begin{tikzcd}
0 \arrow[r] & Y \oplus_pY \arrow[r]                                                & X_1 \oplus_p X'_1 \arrow[r]                                          & \cdots \arrow[r] & X_n \oplus_p X'_n \arrow[r]                                   & Z\oplus_pZ' \arrow[r]                                           & 0 \\
0 \arrow[r] & Y \oplus_pY \arrow[r] \arrow[u, equal] \arrow[d, "S"'] & X_1 \oplus_p X'_1 \arrow[r] \arrow[u, equal] \arrow[d] & \cdots \arrow[r] & PB(\Delta) \arrow[r] \arrow[u] \arrow[d, equal] & Z \arrow[u, "\Delta"'] \arrow[r] \arrow[d, equal] & 0 \\
0 \arrow[r] & Y \arrow[r]                                                          & PO(S) \arrow[r]                                                      & \cdots \arrow[r] & PB(\Delta) \arrow[r]                                          & Z \arrow[r]                                                     & 0
\end{tikzcd}
        \end{equation}
        The lower Yoneda $n$-extension of Diagram (\ref{Baersum}) is called the Baer sum of $[\mathfrak{X}]$ and $[\mathfrak{X}']$.
        \item[ii)] \textbf{Scalar product.} For $\lambda \in \mathbb R$ and $[\mathfrak{X}] \in \Ext^n_{\pBan}(Z,Y)$, let us make pull-back with the operator $\lambda\Id:Z \to Z$:
        \begin{equation}
            \label{scalarproduct}
            \begin{tikzcd}
0 \arrow[r] & Y \arrow[r] \arrow[d, equal] & X_1 \arrow[r] \arrow[d, equal] & \cdots \arrow[r] & X_n \arrow[r]                       & Z \arrow[r]                          & 0 \\
0 \arrow[r] & Y \arrow[r]                                & X_1 \arrow[r]                                & \cdots \arrow[r] & PB(\lambda \Id) \arrow[u] \arrow[r] & Z \arrow[u, "\lambda\Id"'] \arrow[r] & 0
\end{tikzcd}
        \end{equation}
        The lower Yoneda $n$-extension of Diagram (\ref{scalarproduct}) is the result of multiplying $[\mathfrak{X}]$ by $\lambda$.
    \end{itemize}

The zero object in $\Ext^n_{\pBan}(Z,Y)$ is the equivalence class of the so-called \emph{trivial} $n$-exact sequence, which is, for $n=1$, the direct sum sequence 
	\[ \begin{tikzcd}
		0 \arrow[r] & Y \arrow[r, "i"] & Y \oplus_p Z \arrow[r, "q"] & Z \arrow[r] & 0
	\end{tikzcd}\]
	and for $n>1$ is 
	\begin{equation}
		\notag
		\begin{tikzcd}
			0 \arrow[r] & Y \arrow[r, equal] & Y \arrow[r, "0"] & \cdots \arrow[r, "0"] & Z \arrow[r, equal] & Z \arrow[r] & 0.
		\end{tikzcd}
	\end{equation}
    \par For $n = 0$, the set $\Ext^0_{\boldsymbol{\pBan}}(Z,Y) = \mathcal{L}_{\pBan}(Z,Y)$ naturally inherits the vector space structure of continuous linear operators. Under the identification of a $0$-extension as an operator $T \in \mathcal{L}_{\pBan}(Z,Y)$, the general operations via pull-backs and push-outs reduce directly to the standard operator addition and scalar multiplication. Moreover, the zero object in $\Ext^0_{\boldsymbol{\pBan}}(Z,Y)$ is the zero operator $0: Z \to Y$.
    \par Recall that all these operations we have already introduced in the category of $p$-Banach spaces can be defined for $n$-extensions of quasi-Banach spaces in the natural way. Therefore, for any $n\geq 0$, the sets $\Ext^n_{\QBan}(Z,Y)$ are also vector spaces.
	\subsection*{Splicing and cutting Yoneda $n$-extensions}
	Two elements $\mathfrak{F} \in \Ext^n_{\boldsymbol{\pBan}}(E,Y)$ and $\mathfrak{X} \in \Ext^m_{\boldsymbol{\pBan}}(Z,E)$ can be spliced through $E$ to get the Yoneda $(n+m)$-extension
	\begin{equation}\notag
		\begin{tikzcd}[column sep=0.6 cm]
			\mathfrak{FX}: 0 \arrow[r] & Y \arrow[r] & F_1 \arrow[r] & \cdots \arrow[r] & F_n \arrow[rd] \arrow[rr] &              & X_1 \arrow[r] & \cdots \arrow[r] & X_n \arrow[r] & Z \arrow[r] & 0 \\
			&             &               &                  &                           & E \arrow[ru] &               &                  &               &             &  
		\end{tikzcd}
	\end{equation}
	denoted by $\mathfrak{F} \mathfrak{X}$. Using the definition of triviality of Yoneda $n$-extension, it is easy to check that if $\mathfrak{F} \sim 0$ or $\mathfrak{X} \sim 0$ then $\mathfrak{F} \mathfrak{X} \sim 0$. Conversely, if $n \geq 2$, then every Yoneda $n$-extension
	\begin{equation}
	\notag
	\begin{tikzcd}
		\mathfrak{X}: & 0 \arrow[r] & Y \arrow[r, "f_0"] & X_1 \arrow[r, "f_1"] & \cdots \arrow[r, "f_{n-1}"] & X_n \arrow[r, "f_n"] & Z \arrow[r] & 0
	\end{tikzcd}
\end{equation}
	can be broken into shorter Yoneda extensions in the following way. Choose $i$ such that $1 < i \leq n$. Since $\mathfrak{X}$ is exact at $X_i$, then $\ker f_i=\Im f_{i-1}:=W_{i-1}$. Hence, we obtain a Yoneda $(n-i)$-extension $\mathfrak{L}$ and a Yoneda $(i-1)$-extension $\mathfrak{R}$ so that $\mathfrak{X}=\mathfrak{L} \mathfrak{R}$:
	\begin{equation}\notag
	\begin{tikzcd}
		\mathfrak{L}: & 0 \arrow[r] & Y \arrow[r, "f_0"] & X_1 \arrow[r, "f_1"] & \cdots \arrow[r, "f_{i-2}"] & X_{i-1} \arrow[r, "f_{i-1}"] & W_{i-1} \arrow[r] & 0 \\
		\mathfrak{R}: & 0 \arrow[r] & W_{i-1} \arrow[r, "i", hook] & X_i \arrow[r, "f_i"] & \cdots \arrow[r, "f_{n-1}"] & X_n \arrow[r, "f_n"] & Z \arrow[r] & 0
	\end{tikzcd}
\end{equation}
	Repeating the argument, it is easy to see that every Yoneda $n$-extension is the \emph{composition} of $n$ strict hort exact sequences, meaning we can construct the following diagram  
	\begin{equation}
	\notag
	\begin{tikzcd}[column sep=0.5cm]
		\mathfrak{X}: & 0 \arrow[r] & Y \arrow[r, "f_0"] & X_1 \arrow[rr, "f_1"] \arrow[rd, "q_1"] & & X_2 \arrow[r, "f_2"] & \cdots \arrow[r, "f_{n-2}"] & X_{n-1} \arrow[rr, "f_{n-1}"] \arrow[rd, "q_{n-1}"] & & X_n \arrow[r, "f_n"] & Z \arrow[r] & 0 \\
		& & & & W_1 \arrow[ru, "j_1"] & & & & W_{n-1} \arrow[ru, "j_{n-1}"] & & &
	\end{tikzcd}
\end{equation}
	where $W_i=\Im f_i=\ker f_{i+1}$ for $1\leq i < n$, and
	\begin{equation}
	\notag
	\begin{tikzcd}
		\mathfrak{X}_1: & 0 \arrow[r] & Y \arrow[r, "f_0"] & X_1 \arrow[r, "q_1"] & W_1 \arrow[r] & 0 \\
		\mathfrak{X}_{i+1}: & 0 \arrow[r] & W_i \arrow[r, "j_i"] & X_{i+1} \arrow[r, "q_{i+1}"] & W_{i+1} \arrow[r] & 0 \\
		\mathfrak{X}_n: & 0 \arrow[r] & W_{n-1} \arrow[r, "j_{n-1}"] & X_n \arrow[r, "f_n"] & Z \arrow[r] & 0
	\end{tikzcd}
\end{equation}
	are strict short exact sequences such that $\mathfrak{X}=\mathfrak{X}_1 \mathfrak{X}_2 \cdots \mathfrak{X}_{n-1} \mathfrak{X}_n$.
    \par In the case that $n=0$ or $m=0$, the splicing operation coincides with the functorial actions of push-out and pull-back, or standard operator composition:
\begin{itemize}
    \item[i)] If $n \ge 1$ and $m = 0$, given $\mathfrak{F} \in \Ext^n_{\boldsymbol{\pBan}}(E,Y)$ and a $0$-extension $\mathfrak{X}=T \in \mathcal{L}(Z,E) = \Ext^0_{\boldsymbol{\pBan}}(Z,E)$, the spliced extension $\mathfrak{F}\mathfrak{X}$ is precisely the pull-back $n$-extension $\mathfrak{F}\mathfrak{X} = \mathfrak{F}T \in \Ext^n_{\boldsymbol{\pBan}}(Z,Y)$.
    \item[ii)] If $n = 0$ and $m \ge 1$, given a $0$-extension $\mathfrak{F} = S \in \mathcal{L}(E,Y) =\Ext^0_{\boldsymbol{\pBan}}(E,Y)$ and $\mathfrak{X} \in \Ext^m_{\boldsymbol{\pBan}}(Z,E)$, the spliced extension $\mathfrak{F}\mathfrak{X}$ is precisely the push-out $m$-extension $\mathfrak{F}\mathfrak{X}=S\mathfrak{X} \in \Ext^m_{\boldsymbol{\pBan}}(Z,Y)$.
    \item[iii)] If $n = 0$ and $m = 0$, given $0$-extensions $\mathfrak{F} = S \in \mathcal{L}(E,Y)$ and $\mathfrak{X} = T \in \mathcal{L}(Z,E)$, the spliced extension $\mathfrak{F}\mathfrak{X}$ reduces to the usual operator composition $\mathfrak{F}\mathfrak{X} = S T \in \mathcal{L}(Z,Y)= \Ext^0_{\boldsymbol{\pBan}}(Z,Y).$
\end{itemize}
The splicing and cutting of Yoneda $n$-extensions of quasi-Banach spaces is completely analogous.
\subsection{Functors} \label{sec:pre_2.5}
We will work mainly with two functors: the operator functor $\mathcal{L}(E,-)$ and the $\operatorname{Ext}^n(E,-)$-functors associated with Yoneda $n$-extensions, for any $n \ge 0$. Given a $p$-Banach space $E$, the covariant functor $\mathcal{L}_{\boldsymbol{\pBan}}(E,-): \boldsymbol{\pBan} \rightsquigarrow \mathbf{V}$ assigns to each $p$-Banach space $X$ the vector space $\mathcal{L}_{\boldsymbol{\pBan}}(E,X)$ of all continuous linear operators from $E$ to $X$, and to each operator $T: X \to Y$ the linear map $T_*: \mathcal{L}(E,X) \to \mathcal{L}(E,Y)$ given by $T_*(S) = T \circ S$. For a quasi-Banach space $E$, the functor $\mathcal{L}_{\boldsymbol{\QBan}}(E,-): \boldsymbol{\QBan} \rightsquigarrow \mathbf{V}$ is defined analogously.

\par On the other hand, for $n \ge 0$, $\operatorname{Ext}^n_{\boldsymbol{\pBan}}(E,-): \boldsymbol{\pBan} \rightsquigarrow \mathbf{V}$ is a functor which assigns to any $p$-Banach space $X$ the vector space $\operatorname{Ext}^n_{\boldsymbol{\pBan}}(E,X)$ formed by all Yoneda $n$-extensions of $E$ by $X$ in $\boldsymbol{\pBan}$. For an operator $T: X \to Y$, the induced linear map $\operatorname{Ext}^n_{\boldsymbol{\pBan}}(E,T): \operatorname{Ext}^n_{\boldsymbol{\pBan}}(E,X) \to \operatorname{Ext}^n_{\boldsymbol{\pBan}}(E,Y)$ is defined as follows:
\begin{itemize}
    \item For $n \ge 1$, it acts by taking the push-out along $T$:
    \begin{equation}
	\notag
	\begin{tikzcd}
		\mathfrak{X}: & 0 \arrow[r] & X \arrow[r] \arrow[d, "T"'] & X_n \arrow[r] \arrow[d, "i_n"] & X_{n-1} \arrow[r] \arrow[d, equal] & \cdots \arrow[r] & X_1 \arrow[r] \arrow[d, equal] & E \arrow[r] \arrow[d, equal] & 0 \\
		T\mathfrak{X}: & 0 \arrow[r] & Y \arrow[r, "i"] & PO \arrow[r] & X_{n-1} \arrow[r] & \cdots \arrow[r] & X_1 \arrow[r] & E \arrow[r] & 0
	\end{tikzcd}
\end{equation}
    \item For $n = 0$, it acts by composition: $\operatorname{Ext}^0_{\boldsymbol{\pBan}}(E,T)(S) = TS = T_*(S)$.
\end{itemize}

\par Analogously, the functors associated with Yoneda $n$-extensions can be defined in $\boldsymbol{\QBan}$, yielding the $\operatorname{Ext}^n_{\boldsymbol{\QBan}}(E,-)$ functors. Let us observe that, in both categories $\boldsymbol{\pBan}$ and $\boldsymbol{\QBan}$, for $n=0$, the functor $\operatorname{Ext}^0_{\boldsymbol{\pBan}}(E,-)$ coincides canonically with the operator functor $\mathcal{L}_{\boldsymbol{\pBan}}(E,-)$, and $\operatorname{Ext}^0_{\boldsymbol{\QBan}}(E,-)$ coincides with $\mathcal{L}_{\boldsymbol{\QBan}}(E,-)$.

\par Finally, we remark that, although we have focused on the covariant case, the contravariant functors $\mathcal{L}(-,E)$ and $\operatorname{Ext}^n(-,E)$ in both categories $\boldsymbol{\pBan}$ and $\boldsymbol{\QBan}$ are defined analogously. In the case of $\operatorname{Ext}^n_{\pBan}(-,E)$ and $\operatorname{Ext}^n_{\QBan}(-,E)$ the functorial action for $n \ge 1$ is given by pull-backs, and for $n=0$ by pre-composition $S \mapsto ST$.

\section{Existence of relative injective objects in $\pBan$}  \label{main}
\par To develop the classical theory of derived functors in a category $\mathcal{C}$, the concepts of injective and projective objects play a central role. However, strictly speaking, there are no non-zero injective or projective objects in the categories $\boldsymbol{\Ban}$, $\boldsymbol{\pBan}$, and $\boldsymbol{\QBan}$ \cite{pothoven}. Nevertheless, this issue can be partially resolved by appealing to the framework of exact categories (see \cite[Chapter 3]{exact_categ}), which naturally corresponds to the traditional usage of the terms ``injective'' and ``projective'' in Banach space theory.
\par In this context, an object $I$ in $\boldsymbol{\pBan}$ (or $\boldsymbol{\QBan}$) is said to be \emph{injective} if for every strict monomorphism $j: Y \to X$ and every operator $T: Y \to I$, there exists an extension $\hat{T}: X \to I$ of $T$ through $j$, that is, $\hat{T}j = T$. Dually, an object $P$ in $\boldsymbol{\pBan}$ (or $\boldsymbol{\QBan}$) is said to be \emph{projective} if for every strict epimorphism $\pi: X \to Z$ and every operator $T: P \to Z$, there exists a lifting $\hat{T}: P \to X$ of $T$ through $\pi$, that is, $\pi \hat{T} = T$.
\par With the previous definitions, the situation is as follows: $\Ban$ has enough projective objects, which are precisely the $\ell_1(\Gamma)$ spaces, and also has enough injective objects, among which we find the $\ell_{\infty}(\Gamma)$ spaces; remarks that a full characterization of injective Banach spaces remains unknown. On the other hand, for any $0<p<1$, the projective objects in $\pBan$ are the $\ell_p(\Gamma)$ spaces, but $\pBan$ contains no non-zero injective objects. Finally, the most restrictive category in this sense is $\QBan$, in which there are neither non-zero projective nor injective objects \cite[2.9]{hmbst}. This implies, in particular, that the classical construction of right derived functors via injective resolutions is unavailable in $\boldsymbol{\pBan}$. To overcome this limitation, our goal in this section is to construct, for any two fixed $p$-Banach spaces, a space that allows us to establish a \emph{relative} derivation process for covariant, left-exact functors on $\boldsymbol{\pBan}$.
	
	\subsection{The construction of relative injective objects}
	\label{sec:relative-injective}
	Fix $E$ a $p$-Banach space and  $n\in \mathbb N$. Given $X$ a $p$-Banach space, we will need to treat $\Ext_{\pBan}^n(E,X)$ as an index set. In such a case, we write $A=\Ext_{\pBan}^n(E,X)$, and an element $\alpha \in A$ represents the class of the $n$-extension
	\begin{equation}
		\label{eq:next}
		\begin{tikzcd}
			0 \arrow[r] & X \arrow[r, "{j_\alpha}"] & {Z_{1,\alpha}} \arrow[r, "f_1"] & {Z_{2,\alpha}} \arrow[r] & \cdots \arrow[r] & {Z_{n,\alpha}} \arrow[r] & E \arrow[r] & 0
		\end{tikzcd}
	\end{equation}
	where $\norm{j_\alpha}=1$. Recall that the space $Z_{1,\alpha}$ is not well defined up to isomorphism. However, for our purposes below it suffices to choose $Z_{1,\alpha}$ \emph{any} $p$-Banach space which can be inserted in the second place of any $n$-extension which is equivalent to $\alpha$. Before going on, we need to introduce a definition:
\begin{defi}\label{def:lp_sum_infinite}
Fix $0 < p \le 1$. Let $I$ be an arbitrary index set and let $\{X_i\}_{i \in I}$ be a family of $p$-Banach spaces. The \emph{$\ell_p$-sum} of the family $\{X_i\}_{i \in I}$, denoted by $\left(\bigoplus_{i \in I} X_i\right)_{p}$, is defined by
\[
	\left(\bigoplus_{i \in I} X_i\right)_{p} = \left\{ (x_i)_{i \in I} \in \prod_{i \in I} X_i : \sum_{i \in I} \|x_i\|_{X_i}^p < \infty \right\},
\]
equipped with the $p$-norm
\[
	\|(x_i)_{i \in I}\|_p = \left( \sum_{i \in I} \|x_i\|_{X_i}^p \right)^{1/p}.
\]
\end{defi}
	\subsubsection*{Construction} Under the previous hypotheses and the notations of the $n$-extension (\ref{eq:next}), consider the space $\left(\oplus_{\alpha \in A} Z_{1,\alpha}\right)_p$. Given $x \in X$ and two distinct elements
	$\alpha, \alpha'\in A$, we write 
	$x_{\alpha, \alpha'}$ for the element in $\left(\oplus_{\alpha \in A} Z_{1,\alpha}\right)_p$ defined as
	$$x_{\alpha, \alpha'}(\alpha'')= \left\{ \begin{array}{cl}
		j_\alpha(x) &   \text{ if } \alpha''=\alpha, \\
		\\ -j_{\alpha'}(x) &  \text{ if }  \alpha''=\alpha', \\
		\\ 0 &  \text{ otherwise,}
	\end{array}
	\right.
	$$
	and let $X_A$ be the closed span of all $x_{\alpha, \alpha'}$ inside $\left(\oplus_{\alpha \in A} Z_{1,\alpha}\right)_p$. Finally, we define
	$$\mathcal{N}_E^n(X)=\frac{\left(\oplus_{\alpha \in A} Z_{1,\alpha}\right)_p}{X_A}.$$
	
	\begin{thm}\label{0}
		Let $E$ be a fixed $p$-Banach space. For every $p$-Banach space $X$, there exists an into isomorphism $\rho^n_X: X \to \mathcal{N}_E^n(X)$ such that for every element $[\mathfrak{Z}] \in \Ext^n_{\pBan}(E,X)$ it holds that $[\rho_X^n\mathfrak{Z}]=0$.
	\end{thm}
	\begin{proof} Given $[\mathfrak{Z}] \in \Ext^n_{\pBan}(E,X)$, we take $\alpha = [\mathfrak{Z}]$ and observe there is a natural into isomorphism 
		\[ \iota_{\alpha}:Z_{1, \alpha} \longrightarrow \mathcal{N}_E^n(X)\]
		defined by the composition of the canonical inclusion of $Z_{1,\alpha}$ into $\left(\oplus_{\alpha \in A} Z_{1,\alpha}\right)_p$ with the quotient map $\left(\oplus_{\alpha \in A} Z_{1,\alpha}\right)_p \to \mathcal{N}_E^n(X)$.
		Moreover, there exists a \emph{unique} into isomorphism $\rho_X^n:X \to \mathcal N_E^n(X)$ defined by any of the compositions
		\begin{equation}
			\notag 
			\begin{tikzcd}
				X \arrow[r, "j_\alpha"] & Z_{1, \alpha} \arrow[r, "\iota_{\alpha}"] & \mathcal{N}_E^n(X)
			\end{tikzcd},
		\end{equation}
		where $\alpha\in A$, since all of them agree on $X$. Hence $\rho_X^n=\iota_{\alpha} j_{\alpha}=\iota_{\alpha'} j_{\alpha'}$ for any $\alpha, \alpha'\in A$.
		Finally, we can construct the following diagram
		\begin{equation}
	\notag
	\begin{tikzcd}[column sep=0.55cm]
		{[\mathfrak{Z}]}: & 0 \arrow[r] & X \arrow[r, "{j_\alpha}"] \arrow[d, "{\rho_X^n}"'] & {Z_{1,\alpha}} \arrow[r, "f_1"] \arrow[d] \arrow[ld, "{\iota_\alpha}"'] & {Z_{2,\alpha}} \arrow[r] \arrow[d, equal] & \cdots \arrow[r] & {Z_{n,\alpha}} \arrow[r] \arrow[d, equal] & E \arrow[r] \arrow[d, equal] & 0 \\
		{[\rho_X^n\mathfrak{Z}]}: & 0 \arrow[r] & {\mathcal{N}_E^n(X)} \arrow[r] & PO \arrow[r] & {Z_{2,\alpha}} \arrow[r] & \cdots \arrow[r] & {Z_{n,\alpha}} \arrow[r] & E \arrow[r] & 0
	\end{tikzcd}
\end{equation}
		where $\iota_{\alpha} j_{\alpha}=\rho_X^n$ just by the definition of the map $\rho_X^n$.
		This is enough to conclude that $[\rho_X^n\mathfrak{Z}]=0$, since $[\rho_X^n\mathfrak{Z}]$ can be broken into a (trivial) short exact sequence and an $(n-1)$-extension. 
	\end{proof}
	
	\subsection{The functor associated to the space}\label{sec:functor_n}
	Fix a $p$-Banach space $E$. For each $n \in \mathbb{N}$, we define the functor
\[
	\mathcal{N}_E^n(-): \boldsymbol{\pBan} \rightsquigarrow \boldsymbol{\pBan}
\]
which assigns the space $\mathcal{N}_E^n(X)$ to each $p$-Banach space $X$, and the operator $\mathcal{N}_E^n(T): \mathcal{N}_E^n(X) \to \mathcal{N}_E^n(Y)$ to any operator $T: X \to Y$, defined as follows. Let $A = \Ext^n_{\boldsymbol{\pBan}}(E,X)$ and $B = \Ext^n_{\boldsymbol{\pBan}}(E,Y)$ serve as index sets. For each $\alpha \in A$, there exists a unique element $\beta(\alpha) \in B$ and an operator $T_{\alpha, \beta(\alpha)}: Z_{1,\alpha} \to W_{1, \beta(\alpha)}$ making the following pushout diagram commute:
\begin{equation}\label{eq:PO-alfa}
	\begin{tikzcd}
		0 \arrow[r] & X \arrow[r, "j_\alpha"] \arrow[d, "T"'] & {Z_{1,\alpha}} \arrow[r] \arrow[d, "{T_{\alpha, \beta(\alpha)}}"] & {Z_{2,\alpha}} \arrow[r] \arrow[d, equal] & \cdots \arrow[r] & {Z_{n,\alpha}} \arrow[r] \arrow[d, equal] & E \arrow[r] \arrow[d, equal] & 0 \\
		0 \arrow[r] & Y \arrow[r, "j_{\beta(\alpha)}"'] & {W_{1,\beta(\alpha)}} \arrow[r] & {Z_{2,\alpha}} \arrow[r] & \cdots \arrow[r] & {Z_{n,\alpha}} \arrow[r] & E \arrow[r] & 0
	\end{tikzcd}
\end{equation}
The family of norm-one operators $\{T_{\alpha, \beta(\alpha)}: \alpha \in A\}$ induces an operator
\[
	\widehat{\mathcal{N}_E^n}(T): \left( \bigoplus_{\alpha \in A} Z_{1,\alpha} \right)_p \longrightarrow \left( \bigoplus_{\beta \in B} W_{1,\beta} \right)_p
\]
given by $\widehat{\mathcal{N}_E^n}(T)((z_\alpha)_{\alpha \in A}) = \big(T_{\alpha, \beta(\alpha)}(z_\alpha)\big)_{\alpha \in A}$. For any $\alpha, \alpha' \in A$, this operator satisfies
\[
	\widehat{\mathcal{N}_E^n}(T)(x_{\alpha,\alpha'}) = \begin{cases}
		T_{\alpha, \beta(\alpha)}(j_\alpha(x)) = j_{\beta(\alpha)}(Tx), & \text{if } \beta = \beta(\alpha), \\
		-T_{\alpha', \beta(\alpha')}(j_{\alpha'}(x)) = -j_{\beta(\alpha')}(Tx), & \text{if } \beta = \beta(\alpha'), \\
		0, & \text{otherwise,}
	\end{cases}
\]
that is, $\widehat{\mathcal{N}_E^n}(T)(x_{\alpha,\alpha'}) = (Tx)_{\beta(\alpha), \beta(\alpha')}$. Consequently, $\widehat{\mathcal{N}_E^n}(T)$ induces the desired operator
\[
	\mathcal{N}_E^n(T): \mathcal{N}_E^n(X) \longrightarrow \mathcal{N}_E^n(Y)
\]
between the corresponding quotient spaces. The proof of the following proposition is straightforward.
	
	\begin{prop}\label{commu}
		Given an operator $T:X \to Y$ between $p$-Banach spaces, for each $n \in \mathbb{N}$, the following square is commutative:
		\begin{equation}
			\label{square1}
			\begin{tikzcd}
				X \arrow[r, "\rho_X^n"] \arrow[d, "T"] & \mathcal{N}^n_E(X) \arrow[d, "\mathcal{N}^n_E(T)"] \\
				Y \arrow[r, "\rho^n_Y"]                & \mathcal{N}^n_E(Y)                                   
			\end{tikzcd}
		\end{equation}
	\end{prop}

\subsection{An even bigger space}
Fix a $p$-Banach space $E$. Given another $p$-Banach space $X$ and $n \in \mathbb{N}$, the space $\mathcal{N}^n_E(X)$ serves to represent the Yoneda extension vector space $\Ext^n_{\boldsymbol{\pBan}}(E,X)$. To construct a single space that simultaneously recovers all the spaces $\Ext^n_{\boldsymbol{\pBan}}(E,X)$ across all degrees $n \in \mathbb{N}$, we need to combine the family of spaces $\{\mathcal{N}^n_E(X)\}_{n=1}^{\infty}$.
Hence we consider the $\ell_p$-direct sum $\left( \bigoplus_{n \in \mathbb{N}} \mathcal{N}^n_E(X) \right)_p$. Given $x \in X$ and $m, m' \in \mathbb{N}$, let $x_{m,m'}$ denote the element in $\left( \bigoplus_{n \in \mathbb{N}} \mathcal{N}^n_E(X) \right)_p$ defined by:
\[
	x_{m,m'}(m'') = \begin{cases}
		\rho_X^m(x) & \text{if } m'' = m, \\
		-\rho_X^{m'}(x) & \text{if } m'' = m', \\
		0 & \text{otherwise.}
	\end{cases}
\]
Finally, we define the space
\begin{equation}\label{bigspace}
	\mathcal{N}_E^X = \frac{\left( \bigoplus_{n \in \mathbb{N}} \mathcal{N}^n_E(X) \right)_p}{V_X},
\end{equation}
where $V_X$ is the closed linear span of all elements of the form $x_{m,m'}$ with $m, m' \in \mathbb{N}$.
\par Observe that for each $n \in \mathbb{N}$, there exist strict monomorphisms $\iota_n: \mathcal{N}_E^n(X) \to \mathcal{N}_E^X$ given by the composition of the canonical inclusion $\mathcal{N}_E^n(X) \to \left( \bigoplus_{n \in \mathbb{N}} \mathcal{N}^n_E(X) \right)_p$ with the canonical quotient map $\left( \bigoplus_{n \in \mathbb{N}} \mathcal{N}^n_E(X) \right)_p \to \mathcal{N}_E^X$. On the other hand, the strict monomorphisms $\rho_X^n: X \to \mathcal{N}_E^n(X)$ satisfy $\iota_n \rho_X^n = \iota_{n'} \rho_X^{n'}$ for all $n, n' \in \mathbb{N}$. Consequently, we can define a strict monomorphism $\rho_X: X \to \mathcal{N}_E^X$ such that $\rho_X = \iota_n \rho_X^n$ for every $n \in \mathbb{N}$. Thus, we arrive at:

\begin{thm}\label{equal0}
	Let $E$ be $p$-Banach space. For any $p$-Banach space $X$ and for any element $[\mathfrak{Z}] \in \Ext^n_{\pBan}(E,X)$ it holds that $[\rho_X \mathfrak{Z}]=0$.
\end{thm}
\begin{proof} It is a direct consequence of Theorem (\ref{0}) and the fact that $\rho_X=\iota_n\rho_X^n$.
\end{proof}
\par Furthermore, using the functors $\mathcal{N}_E^n(-):\pBan \rightsquigarrow \pBan$ introduced in Section \ref{sec:functor_n}, we can produce the following functor:
\begin{align*}
	\mathcal{N}_E^{(-)}:  \pBan & \rightsquigarrow \pBan \\
	X&  \mapsto \mathcal{N}_E^{(-)}(X):=\mathcal{N}_E^X \\
	[T:X \to Y] & \mapsto \mathcal{N}_E^{(-)}(T)=\mathcal{N}_E^T:\mathcal{N}_E^X \to \mathcal{N}_E^Y,
\end{align*}
The map $\mathcal{N}_E^T:\mathcal{N}_E^X \to \mathcal{N}_E^Y$ is defined by just considering the maps $\mathcal N_E^n$ for $n\in \mathbb N$ and defining
\[(\mathcal{N}_E^n)_{n=1}^\infty:\left( \oplus_{n \in \mathbb N} \mathcal{N}^n_E(X) \right)_p \to \left( \oplus_{n \in \mathbb N} \mathcal{N}^n_E(Y) \right)_p. \ \]
Now, we just have to observe that $(\mathcal N_E^n)_{n=1}^\infty$ takes $V_X$ inside $V_Y$. Therefore, $(\mathcal{N}_E^n)_{n=1}^\infty$ induces a map $\mathcal{N}_E^T:\mathcal{N}_E^X \to \mathcal{N}_E^Y$ between the corresponding quotient spaces  (see \ref{equal0}), as we wanted.

\begin{thm}
	Given an operator $T:X \to Y$ between $p$-Banach spaces, the following square commutes:
	\begin{equation}
		\label{square2}
		\begin{tikzcd}
			X \arrow[r, "\rho_X"] \arrow[d, "T"] & \mathcal{N}_E^X \arrow[d, "\mathcal{N}_E^T"] \\
			Y \arrow[r, "\rho_Y"]                & \mathcal{N}_E^Y                                   
		\end{tikzcd}
	\end{equation}
\end{thm}

\section{Derived functors according to Grothendieck} \label{grothendieck}
In this section, we study the derived functors of the covariant, left-exact operator functor $\mathcal{L}(E,-): \boldsymbol{\pBan} \rightsquigarrow \mathbf{V}$ via an axiomatic approach based on universal $\delta$-functors. Although the concepts of $\delta$-functors and universal $\delta$-functors were originally introduced by Grothendieck \cite{tohoku} in the context of abelian categories, these definitions naturally adapt to quasi-abelian categories by relying on the notion of \emph{strict short exact sequences} \cite{quillen1973}. For the sake of completeness, we recall these adapted definitions below.
\begin{defi}
    Let $\mathcal{C}$ and $\mathcal{D}$ two quasi-abelian categories. A \emph{covariant $\delta$-functor} is a pair $\mathcal{G}=((\mathcal{G}_n)_{n=0}^{\infty}, (\delta_n)_{n=0}^{\infty})$ made of:
    \begin{itemize}
        \item a family $(\mathcal{G}_n)_{n=0}^{\infty}$ of additive functors $\mathcal{G}_n:\mathcal{C} \rightsquigarrow \mathcal{D}$;
        \item a family $(\delta_n)_{n=0}^{\infty}$ of morphisms such that, for each $n \geq 0$ and each strict short exact sequence $\begin{tikzcd}[cramped, column sep=1.5em]
		0 \arrow[r] & A \arrow[r] & B \arrow[r] & C \arrow[r] & 0 
	\end{tikzcd}$ in $\mathcal{C}$, the morphism $\delta_n: \mathcal{G}_n(C) \to \mathcal{G}_{n+1}(A)$ is called the \emph{connecting morphism}.
    \end{itemize}
    Moreover, we require the following axioms to be satisfied:
    \begin{itemize}
        \item Whenever we have two strict short exact sequences $\begin{tikzcd}[cramped, column sep=1.5em]
		0 \arrow[r] & A \arrow[r] & B \arrow[r] & C \arrow[r] & 0 
	\end{tikzcd} $ and $\begin{tikzcd}[cramped, column sep=1.5em]
		0 \arrow[r] & A' \arrow[r] & B' \arrow[r] & C' \arrow[r] & 0 
	\end{tikzcd} $ in $\mathcal{C}$ and a morphism between them, the corresponding square
    \begin{equation}
        \notag
        \begin{tikzcd}
\mathcal{G}_n(C) \arrow[d] \arrow[r, "\delta_n"] & \mathcal{G}_{n+1}(A) \arrow[d] \\
\mathcal{G}_n(C') \arrow[r, "\delta_n"]          & \mathcal{G}_{n+1}(A')        
\end{tikzcd}
    \end{equation}
    commutes.
    \item For each strict short exact sequence $\begin{tikzcd}[cramped, column sep=1.5em]
		0 \arrow[r] & A \arrow[r] & B \arrow[r] & C \arrow[r] & 0 
	\end{tikzcd} $ in $\mathcal{C}$ the corresponding long sequence in $\mathcal{D}$ 
    \begin{equation}
        \label{eq:long_exact} 
        \begin{tikzcd}
\cdots \arrow[r] & \mathcal{G}_n(A) \arrow[r] & \mathcal{G}_n(B) \arrow[r] & \mathcal{G}_n(C) \arrow[r, "\delta_n"] & \mathcal{G}_{n+1}(A) \arrow[r] & \cdots
\end{tikzcd}
    \end{equation}
    is a \emph{cochain complex,} that is, the composition of two consecutive morphism is zero.
    \end{itemize}
Additionally, a $\delta$-functor is said to be \emph{exact} if the long sequence \eqref{eq:long_exact} is exact. 
\end{defi}

\begin{defi}
	A covariant $\delta$-functor $\mathcal{G}=((\mathcal{G}_n)_{n=0}^{\infty}, (\delta_n)_{n=0}^{\infty})$ is \emph{universal} if for any other covariant $\delta$-functor $\mathcal{G}'=((\mathcal{G}'_n)_{n=0}^{\infty}, (\delta'_n)_{n=0}^{\infty})$ and any natural transformation $f_0:\mathcal{G}_0 \to \mathcal{G}'_0$, there exists a unique family of natural transformations $g_n: \mathcal{G}_n \to \mathcal{G}'_n$ for $n \ge 0$ commuting with the connecting homomorphisms and such that $g_0=f_0$.
\end{defi}

One of the main tools to establish universality of a $\delta$-functor without appealing to injective resolutions is the concept of an \emph{effaceable functor}.

\begin{defi}
	Let $\mathcal{C}$ and $\mathcal{D}$ be two quasi-abelian categories. An additive functor $\mathcal{F}: \mathcal{C} \rightsquigarrow \mathcal{D}$ is called \emph{effaceable} if for every object $C$ of $\mathcal{C}$ there exists a strict monomorphism $u: C \to E$ such that $\mathcal{F}(u)=0$.
\end{defi}

A more practical characterization of universal $\delta$-functors was established in \cite[Proposition 2.2.1]{tohoku} using effaceability. For our purposes, we only state the sufficient condition for universality to keep the section self-contained:

\begin{prop}[{\cite[Proposition 2.2.1]{tohoku}}]
	Let $\mathcal{C}$ and $\mathcal{D}$ be two quasi-abelian categories, and let $\mathcal{G}=((\mathcal{G}_n)_{n=0}^{\infty}, (\delta_n)_{n=0}^{\infty})$ be a covariant $\delta$-functor. If $\mathcal{G}_n$ is effaceable for every $n \ge 1$, then $\mathcal{G}$ is a universal $\delta$-functor.
\end{prop}
With this setup in hand, we formulate the concept of a right derived functor in the sense of Grothendieck.

\begin{defi}\label{def:delta-functors_derivado}
	Let $\mathcal{C}$ and $\mathcal{D}$ be two quasi-abelian categories, and let $\mathcal{F}: \mathcal{C} \rightsquigarrow \mathcal{D}$ be an additive, covariant, left-exact functor. A \emph{right derived functor} of $\mathcal{F}$ is a universal $\delta$-functor $\mathcal{RF}=((\mathcal{RF}_n)_{n=0}^{\infty}, (\delta_n)_{n=0}^{\infty})$ such that $\mathcal{RF}_0 = \mathcal{F}$. For each $n \ge 0$, the component $\mathcal{RF}_n: \mathcal{C} \rightsquigarrow \mathcal{D}$ is called the \emph{$n$-th right derived functor} of $\mathcal{F}$.
\end{defi}

\begin{rem}
	By the universality property, if a right derived functor of $\mathcal{F}$ exists, it is unique up to a natural isomorphism of $\delta$-functors.
\end{rem}

Analogous definitions can be formulated for contravariant functors or right-exact functors in a dual manner.

\medskip

We now specialize these abstract notions to our setting. To establish the relative derivation process for the left-exact operator functor $\mathcal{L}(E,-): \boldsymbol{\pBan} \rightsquigarrow \mathbf{V}$, the natural candidate is provided by the sequence of Yoneda extension groups. The first essential step is to verify that this collection satisfies the axioms of a $\delta$-functor.

\begin{prop}
	Let $E$ be a $p$-Banach space. The family of covariant functors $\Ext^n_{\boldsymbol{\pBan}}(E,-): \boldsymbol{\pBan} \rightsquigarrow \mathbf{V}$, associated with the Yoneda $n$-extensions of $p$-Banach spaces, forms a covariant exact $\delta$-functor in degrees $n \in \mathbb{N}_0$. Furthermore, $\Ext^0_{\boldsymbol{\pBan}}(E,X) = \mathcal{L}_{\boldsymbol{\pBan}}(E,X)$.
\end{prop}

  We will write $\Ext_{\boldsymbol{\pBan}}(E,-) = (\Ext^n_{\boldsymbol{\pBan}}(E,-))_{n=0}^{\infty}$.
    
\begin{proof}
For each $n \in \mathbb{N}_0$ and for each short exact sequence 
$$\begin{tikzcd}
	\mathfrak{X}: & 0 \arrow[r] & Y \arrow[r, "j"] & X \arrow[r, "q"] & Z \arrow[r] & 0 
\end{tikzcd}$$
in $\boldsymbol{\pBan}$, we consider the following long homology sequence:
\begin{equation}\label{eq:short_pBan}
	\begin{tikzcd}
		& \cdots \arrow[r] \arrow[dl, phantom, ""{coordinate, name=Z}] & {\Ext_{\boldsymbol{\pBan}}^n(E,Z)} \arrow[dll, "\omega^n" description, rounded corners=6pt, to path={ -- ([xshift=6.5ex]\tikztostart.east) |- (Z) [very near end] \tikztonodes -| ([xshift=-3ex]\tikztotarget.west) -- (\tikztotarget)}] \\
		{\Ext_{\boldsymbol{\pBan}}^{n+1}(E,Y)} \arrow[r] & {\Ext_{\boldsymbol{\pBan}}^{n+1}(E,X)} \arrow[r] \arrow[dl, phantom, ""{coordinate, name=ZZ}] & {\Ext_{\boldsymbol{\pBan}}^{n+1}(E,Z)} \arrow[dll, "\omega^{n+1}" description, rounded corners=6pt, to path={ -- ([xshift=6.5ex]\tikztostart.east) |- (ZZ) [very near end] \tikztonodes -| ([xshift=-3ex]\tikztotarget.west) -- (\tikztotarget)}] \\
		{\Ext_{\boldsymbol{\pBan}}^{n+2}(E,Y)} \arrow[r] & \cdots &
	\end{tikzcd}
\end{equation}
In diagram \eqref{eq:short_pBan}, the connecting homomorphisms $\omega^n: \Ext^n_{\boldsymbol{\pBan}}(E,Z) \to \Ext^{n+1}_{\boldsymbol{\pBan}}(E,Y)$ are defined by splicing an extension class $[\mathfrak{Z}] \in \Ext^n_{\boldsymbol{\pBan}}(E,Z)$ with the short exact sequence $\mathfrak{X}$, that is, $\omega^n([\mathfrak{Z}]) = [\mathfrak{X} \mathfrak{Z}]$. The exactness of the long homology sequence (\ref{eq:short_pBan}) is a standard consequence of the Yoneda extension theory in exact categories (see \cite[Proposition 5.20]{exact_categ}), and the isomorphism $\Ext^0_{\boldsymbol{\pBan}}(E,X) = \mathcal{L}_{\boldsymbol{\pBan}}(E,X)$ follows directly from the definition of $0$-extensions as operators in $\boldsymbol{\pBan}$.   

\par Finally, to verify the naturality of $\omega^n$, consider a morphism of short exact sequences $(\alpha, \beta, \gamma): \mathfrak{X} \to \mathfrak{X}'$ in $\boldsymbol{\pBan}$:
\begin{equation}\notag
	\begin{tikzcd}
		\mathfrak{X}: & 0 \arrow[r] & Y \arrow[r, "j"] \arrow[d, "\alpha"'] & X \arrow[r, "q"] \arrow[d, "\beta"] & Z \arrow[r] \arrow[d, "\gamma"] & 0 \\
		\mathfrak{X}': & 0 \arrow[r] & Y' \arrow[r, "j'"'] & X' \arrow[r, "q'"'] & Z' \arrow[r] & 0
	\end{tikzcd}
\end{equation}
We must show that the following square commutes:
\begin{equation}\notag
	\begin{tikzcd}
		{\Ext_{\boldsymbol{\pBan}}^n(E,Z)} \arrow[r, "\omega^n"] \arrow[d, "\gamma_*"'] & {\Ext_{\boldsymbol{\pBan}}^{n+1}(E,Y)} \arrow[d, "\alpha_*"] \\
		{\Ext_{\boldsymbol{\pBan}}^n(E,Z')} \arrow[r, "\omega'^n"] & {\Ext_{\boldsymbol{\pBan}}^{n+1}(E,Y')}
	\end{tikzcd}
\end{equation}
Let $[\mathfrak{Z}] \in \Ext^n_{\boldsymbol{\pBan}}(E,Z)$ be represented by the $n$-extension 
\[ \mathfrak{Z}: 0 \to Z \to Z_n \to \cdots \to Z_1 \to E \to 0.\]
Evaluating both compositions on $[\mathfrak{Z}]$ yields:
\begin{itemize}
	\item $(\alpha_* \circ \omega^n)([\mathfrak{Z}]) = \alpha_*([\mathfrak{X} \mathfrak{Z}]) = [\alpha(\mathfrak{X} \mathfrak{Z})]$,
	\item $(\omega'^n \circ \gamma_*)([\mathfrak{Z}]) = \omega'^n([\gamma\mathfrak{Z}]) = [\mathfrak{X}' (\gamma\mathfrak{Z})]$.
\end{itemize}
Since the push-out operation along $\alpha$ acts exclusively on the first object $Y$ of $\mathfrak{X}$, push-out and splicing commute, giving $\alpha(\mathfrak{X} \mathfrak{Z}) \equiv (\alpha\mathfrak{X}) \mathfrak{Z}$. Furthermore, by the universal property of the push-out $\alpha\mathfrak{X}$, the map $\beta: X \to X'$ induces a unique morphism of short exact sequences $(\Id_{Y'}, \overline{\beta}, \gamma): \alpha\mathfrak{X} \to \mathfrak{X}'$. Splicing this morphism with $\mathfrak{Z}$ yields a morphism of $(n+1)$-extensions from $(\alpha\mathfrak{X}) \mathfrak{Z}$ to $\mathfrak{X}' \mathfrak{Z}$. Combined with the canonical morphism $\mathfrak{X}' \mathfrak{Z} \to \mathfrak{X}' (\gamma\mathfrak{Z})$ induced by $\gamma$, we obtain a chain of morphisms of $(n+1)$-extensions:
\begin{equation}\notag
	(\alpha\mathfrak{X}) \mathfrak{Z} \longrightarrow \mathfrak{X}' \mathfrak{Z} \longrightarrow \mathfrak{X}' (\gamma\mathfrak{Z})
\end{equation}
which restrict to $\Id_{Y'}$ on $Y'$ and $\Id_E$ on $E$. Consequently, $(\alpha\mathfrak{X}) \mathfrak{Z}$ and $\mathfrak{X}' (\gamma\mathfrak{Z})$ represent the same Yoneda extension class in $\Ext^{n+1}_{\boldsymbol{\pBan}}(E,Y')$. Hence, $\alpha_* \omega^n = \omega'^n \gamma_*$, completing the proof.
\end{proof}

\begin{prop}\label{th:eff}
	Fix $E$ a $p$-Banach space. For each $n \in \mathbb N$, the functor $\Ext^n_{\pBan}(E,-): \pBan \rightsquigarrow \textbf{V}$ is effaceable.
\end{prop}

\begin{proof}
	For each $p$-Banach space $X$, by Theorem \ref{equal0}, it is enough to take $I=\mathcal{N}_X^E$ and $u=\rho_X:X \to \mathcal{N}_X^E$.
\end{proof}

\begin{thm}\label{th:der_universal}
	The $\delta$-functor 
    $$\Ext_{\pBan}(E,-)=(\Ext^n_{\pBan}(E,-))$$ 
    is universal. Consequently, $\Ext_{\pBan}(E,-)$ is the right derived functor of $\mathcal L_{\pBan}(E,-)$.
\end{thm}


\section{Classical (relative) derivation process of the functor $\mathcal{L}(E,-): \boldsymbol{\pBan} \rightsquigarrow \mathbf{V}$} \label{derived}

Fix $E$ a $p$-Banach space. We now show that for any other $p$-Banach space $X$, the space $\mathcal{N}_E^X$ defined by the functor $\mathcal{N}_E^{(-)}: \boldsymbol{\pBan} \rightsquigarrow \boldsymbol{\pBan}$ plays the role of a \emph{relative} injective $p$-Banach object. The word \emph{relative} emphasizes that $\mathcal{N}_E^X$ behaves as an injective object specifically with respect to the pair of spaces $(E,X)$. In other words, the derivation process for the operator functor $\mathcal{L}(E,-)$ must be carried out taking into account both $E$ and the target space $X$ on which the functor acts.

\par Given a short exact sequence $\mathfrak{X}:$
$\begin{tikzcd}[cramped, column sep=1.5em]
	0 \arrow[r] & Y \arrow[r, "j"] & X \arrow[r, "\pi"] & E \arrow[r] & 0
\end{tikzcd}$
in $\boldsymbol{\pBan}$, we can construct a commutative diagram of the form:
\begin{equation}\label{pbpb}
	\begin{tikzcd}
		{[\mathcal{N}_E^Y]}: & 0 \arrow[r] & Y \arrow[r, "{\rho_Y}", hook] & {\mathcal{N}_E^Y} \arrow[r] & {Q_E^Y} \arrow[r] & 0 \\
		{[\mathfrak{X}]=[\mathcal{N}_E^Y \theta]}: & 0 \arrow[r] & Y \arrow[r, "j", hook] \arrow[u, equal] & X \arrow[r, "\pi"] \arrow[u, "\iota", hook] & E \arrow[r] \arrow[u, "\theta", hook] & 0
	\end{tikzcd}
\end{equation}
The upper row of \eqref{pbpb} is called an \emph{$\mathcal{N}_E$-injective presentation of $Y$}.

Repeating this construction with the quotient space $Q_E^Y$ in place of $Y$, we obtain the spliced diagram:
\begin{equation}\notag
	\begin{tikzcd}
		0 \arrow[r] & Y \arrow[r, "\rho_Y"] & \mathcal{N}_E^Y \arrow[rd, "q^1_{Y}"] \arrow[rr] & & \mathcal{N}_E^{Q_E^Y} \arrow[r, "q_Y^2"] & Q_E^{Q_E^Y} \arrow[r] & 0 \\
		& & & Q_E^Y \arrow[ru, "\rho_Y^1", hook] & & &
	\end{tikzcd}
\end{equation}
where we set $Q_E^{Y,2} := Q_E^{Q_E^Y}$ to simplify notation. Iterating this procedure with successive quotient spaces $Q_E^{Y,k} := Q(Q_E^{Y,k-1})$, we arrive at the diagram:
\begin{equation}\label{res}
	\begin{tikzcd}
		0 \arrow[r] & Y \arrow[r, "\rho_Y"] & {\mathcal{N}_E^Y} \arrow[rd, "q_Y^1"] \arrow[rr] & & {\mathcal{N}_E^{Q_E^Y}} \arrow[rd, "q_Y^2"] \arrow[rr] & & {\mathcal{N}_E^{Q_E^{Y,2}}} \arrow[r] & \cdots \\
		& & & {Q_E^Y} \arrow[ru, "\rho_Y^1", hook] & & {Q_E^{Y,2}} \arrow[ru, "\rho_Y^2", hook] & &
	\end{tikzcd}
\end{equation}
A diagram of the form \eqref{res} is called an \emph{$\mathcal{N}_E$-injective resolution of $Y$}, and will be denoted by $Y \to \mathcal{N}_E^Y$.

\subsection{(Relative) derived functor of the functor $\mathcal{L}(E,-)$ in $\pBan$}
Let us consider the additive, covariant left-exact functor $\mathcal{L}(E,-):\pBan \rightsquigarrow \boldsymbol{V}$. Now, for any other $p$-Banach space $X$, we know that its $\mathcal{N}_E$-injective resolution can always be constructed. Then we can compute the classical derived functors of $\mathcal{L}(E,-)$ using $\cN_E$-injective resolutions instead of injective resolutions. The process is mostly analogously to the classical one so we will only sketch it.
\par Let us consider a $p$-Banach space $X$ and the $\mathcal{N}_E$-injective resolution $X \to \mathcal{N}_E^X$ of $X$. If we apply the functor $\mathcal{L}(E,-)$ on all the terms of $X \to \mathcal{N}_E^X$ we obtain the following diagram:
\begin{equation}
	\label{compx}
	\begin{tikzcd}[column sep=0.09cm]
		0 \arrow[rrr] &&& \cL(E,X) \arrow[rrr, "\rho_{X*}"] &&& {\cL(E,\mathcal{N}_E^X)} \arrow[rd, "q_{X*}^1"] \arrow[rr, "d_{X*}^1"] &                                             & {\cL(E,\mathcal{N}_E^{Q_E^X})} \arrow[rd, "q_{X*}^2"] \arrow[rr, "d_{X*}^2"] &                                             & {\cL(E,\mathcal{N}_E^{Q_E^{X,2}})}  \arrow[rrr] &&& \cdots \\ & & & &
		&                       &                                                   & {\cL(E,Q_E^X)} \arrow[ru, "\rho_{X*}^1", hook] &                                                                & {\cL(E,Q_E^{X,2})} \arrow[ru, "\rho_{X*}^2", hook] &                                               
	\end{tikzcd}
\end{equation}
Since the functor $\cL(E,-)$ is not right-exact, Diagram (\ref{compx}) is a cochain complex, and we will denote it $\cL(E, \mathcal{N}_E^X)^{\bullet}$. Therefore, the $n$-th right derived functor of $\cL(E,-)$ is then defined as the $n$-th homology group of the cochain complex (\ref{compx}):
\begin{equation}
	H^n(\cL(E, \mathcal{N}_E^X)^{\bullet})=\frac{\ker d_{n+1*}}{\Im d_{n*}}
\end{equation}
We define the $n$th-right derived functor of $\mathcal{L}(E,-)$ as 
\begin{align*}
	\mathcal{R}^n\cL(E, -):  \pBan & \rightsquigarrow \textbf{V} \\
	X& \mapsto \mathcal{R}^n\cL(E,X):=H^n(\cL(E, \mathcal{N}_E^X)^{\bullet})\\
	[T:X\to Y]& \mapsto\mathcal{R}^n \cL(E, T): H^n(\cL(E, \mathcal{N}_E^X)^{\bullet})\longrightarrow H^n(\cL(E, \mathcal{N}_E^Y)^{\bullet}),
\end{align*}
where the map $\mathcal{R}^n \cL(E, T): H^n(\cL(E, \mathcal{N}_E^X)^{\bullet})\longrightarrow H^n(\cL(E, \mathcal{N}_E^Y)^{\bullet})$ is defined as follows. Let $X \to \mathcal{N}_E^X$ and $Y \to \mathcal{N}_E^Y$ be $\mathcal{N}_E$-injective resolutions of $X$ and $Y$, respectively. Since $\mathcal{N}_E^{(-)}$ is a functor we can define a map $\mathcal{N}_E^X \to \mathcal{N}_E^Y$, and by Proposition \ref{commu}, each square of the following diagram is commutative:

\begin{equation}\label{prev}
	\begin{tikzcd}
		0 \arrow[r] & X \arrow[r, "\rho_X"] \arrow[dd, "T"] & \mathcal{N}_E^X \arrow[rd, "q_X^1"] \arrow[rr, "d^1_X"] \arrow[dd, "T_1"] &                                                                                & \mathcal{N}_E^{Q_E^X} \arrow[rd, "q^2_X"] \arrow[rr, "d^2_X"] \arrow[dd, "T_2"] &                                                                                & \mathcal{N}_E^{Q_E^{X,2}} \arrow[r] \arrow[dd, "T_3"] & \cdots \\
		&                                       &                                                                            & Q_E^X \arrow[ru, "\rho^1_X", hook] \arrow[dd, "\overline{T_1}", near start, dashed] &                                                                                       & Q_E^{X,2} \arrow[ru, "\rho^2_X", hook] \arrow[dd, "\overline{T_2}", dashed, near start] &                                                         &        \\
		0 \arrow[r] & Y \arrow[r, "\rho_Y"]                 & \mathcal{N}_E^Y \arrow[rd, "q^1_Y"] \arrow[rr, "d^1_Y", near start]                   &                                                                                & \mathcal{N}_E^{Q_E^Y} \arrow[rd, "q^2_Y"] \arrow[rr, "d^2_Y", near start]                   &                                                                                & \mathcal{N}_E^{Q_E^{Y,2}} \arrow[r]                   & \cdots \\
		&                                       &                                                                            & Q_E^Y \arrow[ru, "\rho^{1}_Y", hook]                                    &                                                                                       & Q_E^{Y,2} \arrow[ru, "\rho^2_Y", hook]                                      &                                                         &       
	\end{tikzcd}
\end{equation}
Let us observe that $T_1=\mathcal{N}_E^{T}$ and $T_i=\mathcal{N}_E^{\overline{T_{i-1}}}$, for $i\geq 2$.
For each $i \in \mathbb{N}$, the functor $\cL(E,-)$ assigns  to each operator $T_i$ the map $\cL(E,T_i)=T_{i*}$. Therefore, we get the diagram:

{\small \begin{equation}\label{prev2}
		\begin{tikzcd}[column sep=0.33cm]
			0 \arrow[r] & \cL(E,X) \arrow[r, "\rho_{X*}"] \arrow[dd, "T_*"] & \cL(E,\mathcal{N}_E^X) \arrow[rd, "q_{X*}^1"] \arrow[rr, "d_{X*}^1"] \arrow[dd, "T_{1*}"] &                                                                                & \cL(E,\mathcal{N}_E^{Q_E^X}) \arrow[rd, "q_{X*}^2"] \arrow[rr, "d_{X*}^2"] \arrow[dd, "T_{2*}"] &                                                                                & \cL(E,\mathcal{N}_E^{Q_E^{X,2}}) \arrow[dd, "T_{3*}"] \cdots \\
			&                                       &                                                                            & \cL(E,Q_E^X) \arrow[ru, "\rho_{X*}^1", hook] \arrow[dd, "\overline{T_{1*}}", near start, dashed] &                                                                                       & \cL(E,Q_E^{X,2}) \arrow[ru, "\rho_{X*}^2", hook] \arrow[dd, "\overline{T_{2*}}", near start, dashed] &                                                         &        \\
			0 \arrow[r] & \cL(E,Y) \arrow[r, "\rho_{Y*}"]                 & \cL(E,\mathcal{N}_E^Y) \arrow[rd, "q_{Y*}^1"] \arrow[rr, "d_{Y*}^1", near start]                   &                                                                                & \cL(E,\mathcal{N}_E^{Q_E^Y}) \arrow[rd, "q_{Y*}^2"] \arrow[rr, "d_{Y*}^2", near start]                   &                                                                                & \cL(E,\mathcal{N}_E^{Q_E^{Y,2}})                     \cdots \\
			&                                       &                                                                            & \cL(E,Q_E^Y) \arrow[ru, "\rho^{1}_{Y*}", hook]                                    &                                                                                       & \cL(E,Q_E^{Y,2}) \arrow[ru, "\rho_{Y*}^2", hook]                                      &                                                         &       
		\end{tikzcd}
\end{equation}}

Of course, since the functor $\mathcal{L}(E,-)$ is not right-exact, both upper and lower rows of Diagram (\ref{prev2}) are cochain complexes, and it is easy to check that the map $T_{n*}$ satisfies 
\[ T_{n*}[\ker d_{X*}^{n+1}] \subseteq \ker d_{Y*}^{n+1} \quad , \quad T_{n*}[\Im d^n_{X*}]\subseteq \Im d_{X*}^n\]
Hence $T_{n*}$ induces a linear map $\mathcal R^n\cL(E,X)$ between $H^n(\cL(E, \cN_E^X)$ and $H^n(\cL(E, \cN_E^Y)$.

\subsection{Recovering the $\Ext^n_{\boldsymbol{\pBan}}$ functors}

Since $\mathcal{L}(E,-)$ is left-exact, it preserves strict monomorphisms, and therefore the maps $\rho_{X*}^n$ are injective. Consequently,
\[
	\ker d_{X*}^{n+1} = \mathrm{Im}\, \rho_{X*}^n = \rho_{X*}^n\big(\mathcal{L}(E,Q_E^X)\big).
\]
On the other hand, $\mathrm{Im}\, d_{X*}^n = \mathrm{Im}(\rho_{X*}^n q_{X*}^n) = \rho_{X*}^n \big[q_{X*}^n\big(\mathcal{L}(E,\mathcal{N}_E^{Q_E^X})\big)\big]$. Hence, the $n$-th homology space admits a representation of the form:
\[
	H^n\big(\mathcal{L}(E, \mathcal{N}_E^X)^{\bullet}\big) = \frac{\mathcal{L}(E,Q_E^X)}{q_{X*}^n\big(\mathcal{L}(E,\mathcal{N}_E^{Q_E^X})\big)} = \frac{\mathcal{L}(E,Q_E^X)}{\sim}.
\]

Let us first identify the first homology space $H^1\big(\mathcal{L}(E, \mathcal{N}_E^X)^{\bullet}\big)$ with $\Ext_{\boldsymbol{\pBan}}(E,X)$. Given an operator $T: E \to Q_E^X$, we can associate the pullback sequence of the diagram:
\[
\begin{tikzcd}
	0 \arrow[r] & X \arrow[r] & \mathcal{N}_E^X \arrow[r] & Q_E^X \arrow[r] & 0 \\
	0 \arrow[r] & X \arrow[u, equal] \arrow[r] & PB \arrow[u] \arrow[r] & E \arrow[u, "T"'] \arrow[r] & 0
\end{tikzcd}
\]
It is straightforward to verify that if $T \sim T'$, the corresponding exact sequences are equivalent by virtue of the pullback splitting criterion. Conversely, given an extension class in $\Ext_{\boldsymbol{\pBan}}(E,X)$ represented by
\[
\begin{tikzcd}
	0 \arrow[r] & X \arrow[r] & Z \arrow[r] & E \arrow[r] & 0
\end{tikzcd}
\]
we can fit it into a pullback diagram:
\[
\begin{tikzcd}
	0 \arrow[r] & X \arrow[r] & \mathcal{N}_E^X \arrow[r] & Q_E^X \arrow[r] & 0 \\
	0 \arrow[r] & X \arrow[u, equal] \arrow[r] & Z \arrow[u, hook, "S"] \arrow[r] & E \arrow[u, hook, "T"'] \arrow[r] & 0
\end{tikzcd}
\]
Furthermore, if $[\mathfrak{Z}']$ is another exact sequence corresponding to an operator $T': E \to Q_E^X$, then $T - T'$ admits a lifting to $\mathcal{N}_E^X$. Hence, the assignment $[\mathfrak{Z}] \mapsto [T]$ is well-defined.

To identify the homology groups for $n > 1$, we make use of the long homology sequence:
\begin{equation}\label{long}
	\begin{tikzcd}[column sep=1cm, row sep=.75cm]
		0 \arrow[r] & \mathcal{L}(E,X) \arrow[r, "\rho_{X*}"] & \mathcal{L}(E,\mathcal{N}_E^X) \arrow[r, "\pi_{X*}"]
		\arrow[dl, phantom, ""{coordinate, name=Z}]
		& \mathcal{L}(E,Q_E^X) \arrow[dll,
		"\omega_0" description, rounded corners=6pt,
		to path={ -- ([xshift=6.5ex]\tikztostart.east)
			|- (Z) [very near end] \tikztonodes
			-| ([xshift=-3ex]\tikztotarget.west)
			-- (\tikztotarget)}] \\
		& \Ext_{\boldsymbol{\pBan}}(E,X) \arrow[r, "\rho_{X*}"] & \Ext_{\boldsymbol{\pBan}}(E, \mathcal{N}_E^X) \arrow[r, "\pi_{X*}"]
		\arrow[dl, phantom, ""{coordinate, name=Z}]
		& \Ext_{\boldsymbol{\pBan}}(E, Q_E^X) \arrow[dll,
		"\omega_1" description, rounded corners=6pt,
		to path={ -- ([xshift=3ex]\tikztostart.east)
			|- (Z) [very near end] \tikztonodes
			-| ([xshift=-3ex]\tikztotarget.west)
			-- (\tikztotarget)}] \\
		& \Ext^2_{\boldsymbol{\pBan}}(E,X) \arrow[r, "\rho_{X*}"] & \Ext^2_{\boldsymbol{\pBan}}(E, \mathcal{N}_E^X) \arrow[r, "\pi_{X*}"]
		\arrow[dl, phantom, ""{coordinate, name=Z}]
		& \Ext^2_{\boldsymbol{\pBan}}(E, Q_E^X) \arrow[dll,
		"\omega_2" description, rounded corners=6pt,
		to path={ -- ([xshift=3ex]\tikztostart.east)
			|- (Z) [very near end] \tikztonodes
			-| ([xshift=-3ex]\tikztotarget.west)
			-- (\tikztotarget)}] \\
		& \cdots
	\end{tikzcd}
\end{equation}
By Theorem~\ref{equal0}, for any element $[\mathfrak{Z}] \in \Ext^{n+1}_{\boldsymbol{\pBan}}(E,X)$, we have $\rho_{X*}([\mathfrak{Z}]) = 0$. Since the sequence \eqref{long} is exact, it follows that $\mathrm{Im}\, \omega_{n+1} = \ker \rho_{X*} = \Ext^{n+1}_{\boldsymbol{\pBan}}(E,X)$ for every $n \in \mathbb{N}$; that is, $\omega_{n+1}$ is surjective. From this fact, we obtain the \emph{dimension reduction formula}:

\begin{thm}[Dimension reduction formula]\label{reductionformula}
	For every $n \ge 1$, there is a natural isomorphism
	\begin{equation}\label{eq:red_formula}
		\Ext_{\boldsymbol{\pBan}}^{n}(E,X) \simeq \Ext^{n-1}_{\boldsymbol{\pBan}}(E,Q_E^X).
	\end{equation}
\end{thm}

Applying \eqref{eq:red_formula} recursively yields the chain of natural isomorphisms:
\begin{align*}
	\Ext^{n}_{\boldsymbol{\pBan}}(E,X) &\simeq \Ext^{n-1}_{\boldsymbol{\pBan}}(E,Q_E^X) \simeq \cdots \simeq \Ext_{\boldsymbol{\pBan}}(E,Q_E^{X,n-1}) \\
	&= \frac{\mathcal{L}(E,Q_E^{X,n})}{q_{X*}^n\big(\mathcal{L}(E, \mathcal{N}_E^{Q_E^X})\big)} = H^n\big(\mathcal{L}(E,\mathcal{N}_E^X)^{\bullet}\big).
\end{align*}

This identity establishes the fundamental link between Yoneda $n$-extensions and the resolution-based derivation process in $\boldsymbol{\pBan}$. Consequently, for every fixed $E \in \boldsymbol{\pBan}$ and $n > 0$, the extension functor $\Ext^n_{\boldsymbol{\pBan}}(E,-)$ is canonically isomorphic to the $n$-th right derived functor $R^n\mathcal{L}(E,-)$ of the left-exact functor $\mathcal{L}(E,-)$. In particular, this identification confirms that the cohomology groups $H^n\big(\mathcal{L}(E,\mathcal{N}_E^X)^{\bullet}\big)$ depend solely on the spaces $E$ and $X$, remaining completely independent of the choice of relative injective resolution.

\section{What happens in $\QBan$?}\label{sec:QBan}
\par In this final section we explore the extendability of the constructions from Section~\ref{sec:relative-injective} to the category $\QBan$. On the one hand, there are no injective objects in $\QBan$, since there are none in $\pBan$ for any $0<p \leq 1$. But, in contrast, since every quasi-Banach space is a p-Banach space for some $0<p \leq 1$, one might naturally expect that the constructions of Section \ref{main} could be carried out in $\QBan$ with minimal adaptations. However, the principal obstacle when we go from $\pBan$ to $\QBan$ is that, in general, (infinite) limits and colimits fail to exist in $\QBan$. To address this obstacle, we divide the construction of a ``relative injective object'' in $\QBan$ into several cases according to the length of the extensions. First, consider the case $n=1$ and let
\[ 
\begin{tikzcd}
    \mathfrak{X}: & 0 \arrow[r] & Y \arrow[r] & X \arrow[r] & Z \arrow[r] & 0 
\end{tikzcd}
\]
be a strict short exact sequence of quasi-Banach spaces. Thus, $Y$ is a $p_1$-Banach space and $Z$ is a $p_2$-Banach space, with $0 < p_1, p_2 \le 1$. The central question now is which value of $p$ can be associated with the middle quasi-Banach space $X$.
By a well-known result of Kalton \cite{Kalton}, there exists a function $f(p_1,p_2)$ such that $X$ is an $f_1(p_1,p_2)$-Banach space. Explicitely, one can take for every $\varepsilon>0$, 
\[ 
f_1(p_1,p_2) = \begin{cases} 
\min\{p_1,p_2\} & \text{if } p_1 \neq p_2, \\ 
p_1 - \varepsilon & \text{if } p_1 = p_2. 
\end{cases} 
\]
Consequently, the construction described in Section~\ref{sec:relative-injective} also works in $\QBan$ for $n=1$. Precisely, given a quasi-Banach space $E$, for any quasi-Banach space $X$ there exists $p \in (0,1]$ such that every middle space of an extension in $\Ext_{\QBan}(E,X)$ is isomorphic to a $p$-Banach space. Therefore, Theorem~\ref{0} can be directly improved as follows:

\begin{thm}
Let $E$ be a quasi-Banach space. For every quasi-Banach space $X$, there exists a quasi-Banach space $\mathcal{N}_E^1(X)$ and an isomorphic embedding $\rho^1_X: X \to \mathcal{N}_E^1(X)$ such that $[\rho_X^1 \mathfrak{Z}] = 0$ for every element $[\mathfrak{Z}] \in \Ext_{\QBan}(E,X)$.
\end{thm}

The main difficulty arises when we consider $n > 1$ and attempt to carry out the same argument: let $X$ be a $p$-Banach space and $Y$ an $r$-Banach space, with $0 < p, r \le 1$. Given $n \in \mathbb{N}$ with $n > 1$, is there a function $f_n(p,r) \in (0,1]$ such that any element in $\Ext^n_{\QBan}(X,Y)$ is the equivalence class of an exact sequence
\begin{equation} \label{eq:long_qban}
\begin{tikzcd}
    \mathfrak{Z}: & 0 \arrow[r] & Y \arrow[r] & Z_1 \arrow[r] & \cdots \arrow[r] & Z_n \arrow[r] & X \arrow[r] & 0 
\end{tikzcd}
\end{equation}
in which $Z_1, \dots, Z_n$ are $f_n(p,r)$-Banach spaces? A partial answer to this question is known provided the space $X$ in Diagram~\ref{eq:long_qban} is separable. Specifically, we have the following result:

\begin{thm}\label{thm:stability-qban}
Let $n \ge 1$, let $X$ be a separable $p_1$-Banach space, let $Y$ be a $p_2$-Banach space, and fix $0 < p < \min\{p_1, p_2\}$. Every element of $\Ext^n_{\QBan}(X,Y)$ has a representative in which all middle spaces are $p$-Banach spaces.
\end{thm}
\begin{proof}
    We will do the proof by induction on the length $n$. For $n=1$ is just Kalton's theorem mentioned above. Let us suppose the result true for $n-1$, and prove it for $n$. Choose $r$ such that $p<r<\min\lbrace p_1, p_2 \rbrace$. Since $X$ is a $p_1$-Banach space, it is also a $r$-Banach space. Moreover, since $X$ is separable, we can take an $r$-projective presentation $X$ using the space $\ell_r(\mathbb{N})$
    \[ 
\begin{tikzcd}
    \mathfrak{P}: & 0 \arrow[r] & \ker \pi \arrow[r] & \ell_r(\mathbb{N}) \arrow[r, "\pi"] & X \arrow[r] & 0 
\end{tikzcd}
\]
in which the kernel $\ker \pi$ is a separable $r$-Banach space. Since $r <p_2$, due to a result to P. Scholze (see \cite[Theorem 1.1]{scholze2} and \cite[Remark 8.15]{scholze1}, cf.  \cite{scholze3}, which is a more suitable reformulation in $\QBan$), we know that $\Ext_{\QBan}(\ell_r(\mathbb{N}),Y)=0$, so that, the exact long homology sequence associated to the projective presentation of $X$ is the following:
\begin{equation}
    \begin{tikzcd}[column sep=1.5em]
\cdots \arrow[r] & {\Ext_{\QBan}^{n-1}(\ker \pi,Y)} \arrow[r, "\delta"] & {\Ext^n_{\QBan}(X,Y)} \arrow[r] & 0 \arrow[r] & {\Ext_{\QBan}^{n}(\ker \pi,Y)} \arrow[r] & \cdots
\end{tikzcd}
\end{equation}
which means that the connecting morphism $\delta$, which is defined by splicing with $[\mathfrak{P}]$, is surjective. That is to say, for every $[\mathfrak{Z}] \in \Ext^n_{\QBan}(X,Y)$ there exists $[\mathfrak{Z'}] \in \Ext_{\QBan}^{n-1}(\ker \pi,Y)$ such that $\delta([\mathfrak{Z'}])=[\mathfrak{Z}]$. Since $p < \min\lbrace r,p_2 \rbrace$, by the induction hypothesis, there exists a representative $[\mathfrak{W}] \in \Ext_{\QBan}^{n-1}(\ker \pi ,Y)$
\begin{equation}
    \notag
    \begin{tikzcd}
    \mathfrak{W}: & 0 \arrow[r] & Y \arrow[r] & W_1 \arrow[r] & \cdots \arrow[r] & W_{n-1} \arrow[r] & \ker \pi \arrow[r] & 0 
\end{tikzcd}
\end{equation}
such that $[\mathfrak{W}]=[\mathfrak{Z'}]$ whose all of its middle spaces $W_j$ are $p$-Banach spaces. Splicing $\mathfrak{W}$ with $\mathfrak{P}$ we obtain the extension
\begin{equation}
    \notag 
\begin{tikzcd}[column sep=1.5em]
{\mathfrak{W'}:} & 0 \arrow[r] & Y \arrow[r] & W_1 \arrow[r] & \cdots \arrow[r] & W_{n-1} \arrow[rd] \arrow[rr] &                     & \ell_r(\mathbb N) \arrow[r] & X \arrow[r] & 0 \\
                   &             &             &               &                  &                           & \ker \pi \arrow[ru] &                             &             &  
\end{tikzcd} 
\end{equation}
which represents $[\mathfrak{Z}]$. Finally, let us observe that, in $\mathfrak{W'}$, all the spaces $W_i$ are $p$-Banach spaces and $\ell_r(\mathbb{N})$ is an $r$-Banach space, with $p <r$, then they are all $p$-Banach spaces.
\end{proof}
The previous result allows an improvement of Theorem \ref{thm:stability-qban} to the following:
\begin{thm}\label{thm:eff_qban}
Let $E$ be a separable quasi-Banach space. For every quasi-Banach space $X$, there exists a quasi-Banach space $\mathcal{N}_E^X$ and an isomorphic embedding $\rho_X: X \to \mathcal{N}_E^X$ such that $[\rho_X \mathfrak{Z}] = 0$ for every element $[\mathfrak{Z}] \in \Ext_{\QBan}^n(E,X)$.
\end{thm}
Finally, using Theorem \ref{thm:eff_qban}, and the easy fact that the long homology sequence \eqref{eq:longhomology} can be analogously defined in $\QBan$, we get:
\begin{thm}
    Consider $E$ be a separable quasi-Banach space. The $\delta$-functor 
    $$\Ext_{\QBan}(E,-)=(\Ext_{\QBan}^n(E,-))$$ 
    is universal. Consequently, $\Ext_{\QBan}(E,-)$ is the right derived functor of $\mathcal{L}_{\QBan}(E,-)$.
\end{thm}

\section*{Acknowledgements}
The author is very thankful to the BANEXT research group in Universidad de Extremadura, and also to Alberto Salguero-Alarcón and Pedro Tradacete for their valuable ideas, comments and suggestions.
\bibliographystyle{siam}

\begin{thebibliography}{99}

    \bibitem{buehler2010} T.~B{\"u}hler,
        \emph{Exact categories},
        Expos. Math. \textbf{28} (2010), no.~1, 1--69.
        
	\bibitem{hmbst} F. Cabello Sánchez, J. M. F. Castillo, \emph{Homological Methods in Banach Space Theory}, Cambridge Studies in Advanced Mathematics, 203, (2023).

	
	\bibitem{exact_categ} L. Frerick, D. Sieg, \emph{Exact categories in functional analysis}, Script (2010).
	
	\bibitem{tohoku} A. Grothendieck, \emph{Sur quelques points d'algèbre homologique}, Tôhoku Math. J., (2), 9 (2): 119–221, (1957). 
	%
	
	\bibitem{Kalton} N. J. Kalton, \emph{Convexity, type and the three-space problem}, Studia Math. (1981), 247-287. 
	
	\bibitem{fspace} N. J. Kalton, N. T. Peck, and J. W. Roberts, \emph{An F-space sampler,} Cambridge University Press, 1984.
	
	
	\bibitem{pothoven} K. Pothoven, \emph{Projective and injective objects in the category of Banach spaces}, Proc. Amer. Math. Soc. 22 (1969), 437--438. 
    
    \bibitem{quillen1973}
        D.~Quillen,
    \emph{Higher algebraic {K}-theory: {I}},
    in: Algebraic {K}-theory, {I}: {H}igher {K}-theories,
    Lecture Notes in Mathematics, vol. 341,
    Springer, Berlin, Heidelberg, 1973, pp. 85--147.

    \bibitem{schneiders1999}
J. P.~Schneiders,
\emph{Quasi-abelian categories and sheaves},
M{\'e}m. Soc. Math. Fr. (N.S.) \textbf{76} (1999), 1--134.

\bibitem{scholze1} P. Scholze, \emph{Lectures on Analytic Geometry}, \href{https://arxiv.org/abs/2605.03655}{arXiv:2605.03655 [math.CT, math.FA, math.NT].}


\bibitem{scholze2} P. Scholze, \emph{Liquid Tensor Experiment}, Exp. Math. 31 (2022), no. 2, 349--354.

\bibitem{scholze3} P. Scholze, \emph{Nonconvexity and discretization}, \href{https://mathoverflow.net/users/6074/peter-scholze}{https://mathoverflow.net/q/386796.}

\end{thebibliography}

\end{document}